\documentclass[10pt]{amsart}

\usepackage{hyperref}

\usepackage{amsmath}
\usepackage{amssymb}
\usepackage{amsthm}
\usepackage{thmtools}

\usepackage{tikz-cd}
\usepackage[utf8]{inputenc}
\usepackage{lipsum}
\usepackage{stmaryrd}
\usepackage[bb=libus]{mathalpha}

\DeclareSymbolFont{textoperators}{OT1}{\familydefault}{m}{n}
\SetSymbolFont{textoperators}{bold}{OT1}{\familydefault}{b}{n}
\makeatletter
\renewcommand{\operator@font}{\mathgroup\symtextoperators}
\makeatother

\usepackage[marginpar]{todo}

\usepackage[%
    backend=bibtex, 
    sortcites, 
    sorting=nty, 
    safeinputenc, 
    citestyle=alphabetic, %
    bibstyle=alphabetic, %
    hyperref=true, 
    maxbibnames=3, 
    maxcitenames=3, 
    url=true, 
    doi=false, 
    giveninits=true,
    ]%
{biblatex}
\usepackage[margin=3cm]{geometry}

\usepackage[capitalize]{cleveref}

\newtheorem{theorem}{Theorem}[section]
\newtheorem{proposition}[theorem]{Proposition}
\newtheorem{lemma}[theorem]{Lemma}
\newtheorem{corollary}[theorem]{Corollary}

\theoremstyle{definition}
\newtheorem{definition}[theorem]{Definition}
\newtheorem{notation}[theorem]{Notation}
\newtheorem{example}[theorem]{Example}

\theoremstyle{remark}
\newtheorem{remark}[theorem]{Remark}

\let\emptyset\varnothing

\newcommand{\ourspace}[1]{\mathcal{M}^{\mathbb{k}}_{#1}}

\begin{document}

\author{Ulrich Bauer}

\author{Cameron Gusel}

\author{Luis Scoccola}

\title[A metrically complete and Krull--Schmidt space of persistence modules]{A metrically complete and Krull--Schmidt space of multiparameter persistence modules}

\begin{abstract}
We show that the observable category of q-tame multiparameter persistence modules satisfies good metric and algebraic properties: it forms a complete metric space with respect to the interleaving distance,
and it is Krull--Schmidt in the sense that every object admits an essentially unique decomposition into indecomposables.
Moreover, we show that these metric and algebraic structures are compatible: two objects are at distance zero if and only if they are isomorphic.
We argue that the observable category of q-tame multiparameter persistence modules is the right setup for multiparameter persistence by showing that many of the categories already considered in the literature form full subcategory of this category.
We also characterize precompact sets in terms of finite representation type of certain discretizations, and show that the image of several of the main constructions in multiparameter persistence is precompact.
\end{abstract}

\maketitle

\tableofcontents

\section{Introduction}

\subsection{Motivations and context}

A central concept in topological inference and topological data analysis is persistent homology, which arises from applying homology to a filtration of topological spaces indexed by one or more parameters.
The output is a \emph{persistence module}, that is, a functor from the indexing poset $\mathbb{R}^n$ to the category of vector spaces; or, in the language of representation theory, a representation of the poset $\mathbb{R}^n$.
\emph{Persistence theory} is distinct from the representation theory of posets, since it studies the algebraic properties of persistence modules as well as their metric properties, enabled by the existence of a well-behaved distance between persistence modules, known as the \emph{interleaving distance} \cite{chazal-cohen-steiner-glisse-guibas-oudot,lesnick}.
Applications of persistence theory span a range of scientific disciplines such as
neuroscience and biology
\cite{schneider-et-al,gardner-et-al,rizvi-et-al,benjamin-et-al},
as well as pure mathematics, specifically in geometry and analysis
\cite{biran-cornea-zhang,shelukhin,polterovich-shelukhin,usher-zhang,buhovsky-et-al,buhovsky-2,polterovich-rosen-samvelyan-zhang}.

\subsection*{One-parameter persistence}
The theory of \emph{one-parameter persistence} (i.e., the case $n=1$) is well-developed.
Commonly, the parameter of the filtration is interpreted as a geometric scale, and the resulting one-parameter persistence modules are used to identify topological features of the data that persist over a range of scales.
Applications of one-parameter persistence in both science and pure mathematics rely on two fundamental results:
\begin{itemize}
  \item[(R1)] The decomposition theorem (also known as the structure theorem), which states that one-parameter persistence modules decompose uniquely into interval modules \cite{crawleyboevey14}.
  Such intervals form the \emph{barcode} of the module.
  \item[(R2)] The stability theorem (also known as the isometry theorem), which states that the interleaving distance between one-parameter persistence modules equals a certain optimal matching distance (the \emph{bottleneck distance}) between their barcodes \cite{chazal,bauer-lesnick}.
\end{itemize}
Neither of these results hold for arbitrary one-parameter persistence modules; some tameness assumptions are required, which led to increasingly more general versions (finitely presented \cite{zomorodian-carlsson,cohen-steiner-edelsbrunner-harer,chazal-cohen-steiner-glisse-guibas-oudot}, pointwise finite-dimensional \cite{crawleyboevey14,bauer-lesnick}, q-tame \cite{chazal}).

\subsection*{The observable category of q-tame modules}
The \emph{observable category of q-tame modules} \cite{chazal} provides
the most general framework in which strong algebraic and metric theories have been developed.
Briefly, a \emph{q-tame} (potentially multiparameter) persistence module is one for which the structure maps between indices with strict inequality in all components have finite rank;
virtually all relevant examples of persistence modules arising in topological data analysis are q-tame.
The letter ``q'' in ``q-tame'' stands for ``quadrant''; see \cite{chazal} for an explanation of this terminology in the one-parameter case.
The \emph{observable category} is obtained by taking a quotient by the \emph{ephemeral modules}, that is, 
those whose structure maps between indices with strict inequality in all components are zero,
or equivalently, those that are at interleaving distance zero from the zero module \cite{harsu2024}.

The observable category of q-tame one-parameter persistence modules has the following crucial properties:
\begin{itemize}
  \item[(P1)] The category is Abelian and Krull--Schmidt, in the sense that every object decomposes as a direct sum of indecomposables in an essentially unique way.
  \item[(P2)] Two persistence modules are isomorphic if and only if they are at interleaving distance zero.
  \item[(P3)] The extended metric induced by the interleaving distance is complete, in the sense that every Cauchy sequence of persistence modules has a limit.
\end{itemize}
For references, see \cite[Theorems~3.9~and~4.5]{chazal} for (P1) and (P2) and \cite[Theorem~7]{bubenik-vergili} for (P3).
Together, these properties provide a theory of persistence modules in which both algebraic (P1) and geometric (P3) techniques can be applied in a compatible way (P2).


It may not be clear at first why requiring q-tameness
and passing to the observable category are necessary.
We give some examples to illustrate this point.
Without passing to the observable category, property (P2) is clearly not satisfied; for example, the interval modules corresponding to the intervals $(0,1)$ and $[0,1]$ are at interleaving distance zero, yet not isomorphic.
Even in the observable setting, both properties (P1) and (P2) fail without the assumption of q-tameness; see \cref{example:no-decomposition} and \cite[Examples~4.6]{chazal}, respectively.
In the collection of all (potentially multiparameter) persistence modules, the closure of the collection pointwise finite-dimensional modules is the collection of q-tame modules, which further motivates the notion of q-tameness as a way to also achieve property (P3).

Importantly, the proofs of (P1) and (P2) referred to above rely on results (R1) and (R2), which do not generalize to multiparameter persistence modules, as we now discuss.

\subsection*{Multiparameter persistence}
The one-parameter persistence approach to topological inference is robust to small perturbations of the data, caused by sampling and bounded noise; however, it is not robust to unbounded noise such as outliers.
To address this issue, the use of \emph{multiparameter persistence} ($n\geq 2$) has been proposed \cite{carlsson-zomorodian}, where the additional parameters can be used to control the influence of outliers \cite{blumberg2022stability2parameterpersistenthomology} or to encode additional information about the data \cite{carlsson-zomorodian}; see \cite{botnan23} for an introduction.

From the algebraic viewpoint, multiparameter persistence is significantly more complicated in both theory and practice.
On the theoretical side, result (R1) does not generalize, in the sense that the non-linear posets $\mathbb{R}^n$ with $n\geq 2$ have wild representation type \cite{nazarova} (meaning that their indecomposable representations cannot be effectively classified), and (R2) fails as well \cite{bauer23}, since indecomposables form a dense set in the interleaving distance.
On the practical side, the interleaving distance is NP-hard to compute \cite{bjerkevik-botnan-kerber}, and decomposition into indecomposables is not, as of yet, known to be computable in cubic time.
Nevertheless, several recent works have made substantial progress regarding the
structure \cite{botnan23,botnan-oppermann-oudot,brustle-oudot-scoccola-thomas},
stability \cite{blumberg2022stability2parameterpersistenthomology,botnan-oppermann-oudot-scoccola},
and computability of multiparameter persistence \cite{lesnick-wright,morozov-scoccola,dey-jendrysiak-kerber}.

\subsection*{Tameness assumptions}
Mimicking the history of one-parameter persistence, the above-mentioned foundational works in multiparameter persistence typically assume that the modules in consideration are finitely presented or otherwise parametrized by discrete posets.

Although finitely presented multiparameter persistence modules satisfy properties (P1) and (P2), they do not satisfy property (P3) and are 
a quite restrictive class of modules, and indeed more restrictive than in the one-parameter case: for example, while a classical Morse function on a compact manifold has finitely presented sublevel set one-parameter persistent homology, the sublevel set two-parameter persistent homology of a bivariate Morse function is usually not finitely presented
(see, e.g., the examples in \cite{budney-kaczynski}).

A more general notion of tameness for multiparameter persistence is that of Miller \cite{miller}, which we call \emph{m-tame}, to avoid confusion with other notions of tameness (e.g., \cite{chacholski-et-al,chacholski-et-al-2}).
An \emph{m-tame} $n$-parameter persistence module is a functor $M : \mathbb{R}^n \to \mathbf{Vect}$ that admits a finite upset presentation (that is, such that $M \cong \operatorname{cok}(Q \to P)$, with $P$ and $Q$ decomposing as indicator modules of upsets); see \cite[Theorem~6.12(4)]{miller}.
The notion of m-tame module is more general than that of finitely presented module,
and it encompasses, for example, the sublevel set persistence of two-dimensional Morse functions \cite[Corollary~2.11]{budney-kaczynski}, as well as subanalytically constructible sheaves \cite[Theorem~$4.5^\prime$]{miller-2}.
However, m-tame modules do not satisfy properties (P2) and (P3), and the notion is still restrictive: for example, an m-tame module necessarily restricts to a one-parameter persistence module with finitely many bars over any one-dimensional slice.
Moreover, m-tame modules do not satisfy property (P1) in the sense that they do not form an Abelian subcategory of the category of all multiparameter persistence modules \cite{miller,waas}.

A yet more general notion is that of pointwise finite-dimensional modules.
These satisfy property (P1) \cite{botnan19}, while property (P2) does not hold on the nose, but can be achieved by passing to the observable category. 
Even then, however, property (P3) does not hold since any q-tame module can be approximated by pointwise finite-dimensional modules \cite{chazal}.

This leaves us with q-tameness as a candidate tameness condition for a well-behaved multiparameter persistence theory. 
It is worth pointing out that the extension of q-tameness to the multiparameter case as introduced in this paper has not been considered before in the literature, but is the natural candidate complementing the notion of ephemeral multiparameter persistence module \cite{berkouk-petit,harsu2024}.
A main contribution of this paper is to 
establish properties (P1) to (P3) for q-tame multiparameter persistence modules.
This is not a trivial extension of the one-parameter case, since, as stated above, the proofs of (P1) and (P2) in the one-parameter case rely on results (R1) and (R2), which do not generalize to multiparameter persistence.

\subsection*{Metric properties}
Beyond metric completeness, which implies the existence of limiting objects with well-defined algebraic properties, other metric properties are convenient when working with persistence modules.
In the case of one-parameter persistence, several works have studied such properties
\cite{blumberg-gal-mandell-pancia,bubenik-vergili,perea-munch-firas,chowdhury,bubenik-desilva-nanda,che-galaz-guijarro-membrillo}.
This includes the study of separability \cite[Theorem~4.18]{bubenik-vergili}, which, together with completeness, implies that the observable category of q-tame one-parameter persistence modules is Polish, which is relevant for defining probability measures on persistence modules.
And it includes as well the characterization of precompact sets \cite[Theorem~3.7]{perea-munch-firas}, which is relevant for approximation of continuous functions on the space of persistence modules.
This analysis has not been carried in the multiparameter case.

As we show in this paper, once a suitable space of multiparameter persistence modules is identified,  precompact sets are characterized by the fact that all their discretizations have finite representation type.
In this sense, a metric condition (in this case precompactness) allows one to control the representation type.


\subsection{Contributions}

\begin{theorem}
  \label{theorem:main-theorem}
  Let $n \geq 1$.
  The observable category of q-tame $n$-parameter persistence modules satisfies the following properties:
  \begin{enumerate}
  \item[(P1)] The category is Abelian and Krull--Schmidt, in the sense that every object decomposes as a direct sum of indecomposables in an essentially unique way.
  \item[(P2)] Two persistence modules are isomorphic if and only if they are at interleaving distance zero.
  \item[(P3)] The extended metric induced by the interleaving distance is complete, in the sense every Cauchy sequence of persistence modules has a limit.
  \end{enumerate}
\end{theorem}

We believe that
the observable category of q-tame $n$-parameter persistence modules
provides a sufficiently general setting for multiparameter persistence due to the following, which states that many standard categories of persistence modules of interest are contained isometrically as full subcategories.

\begin{remark}
  \label{remark:main-subcategories}
  The observable category of q-tame $n$-parameter persistence modules includes, up to equivalence, the following as full subcategories:
  \begin{itemize}
    \item The category of finitely presented $n$-parameter persistence modules.
    \item The category of $n$-parameter persistence modules that can be obtained as the sublevel set persistence of some continuous function $f : X \to \mathbb{R}^n$ with $X$ a finite triangulable space.
    \item The category of $2$-parameter persistence modules that can be obtained by applying Degree-Rips homology or Subdivision-Rips homology to some finite pseudometric space.
  \end{itemize}
  The inclusions are given by mapping a persistence module to its representative in the observable category; in particular, the inclusions respect the interleaving distance.
\end{remark}

In particular, one can consider the metric closure of the (set of isomorphism classes of) objects of any of the categories in \cref{remark:main-subcategories}, and get a category of persistence modules satisfying properties (P1) to (P3).
In the case of the category of finitely presented persistence modules, this closure is used in \cite{bauer23} to show that indecomposability is a generic property in the closure of the space of finitely presented multiparameter persistence modules (see \cite[Remark~1.5]{bauer23}).
We give an explicit characterization of the metric closure
of the set of (isomorphism classes of) finitely presented persistence modules in \cref{corollary:characterization-closure-finitely-presented}, and denote it by $\mathbf{bg}_n$.

For convenience, let us denote the extended metric space given by the isomorphism classes of observable q-tame $n$-parameter persistence modules endowed with the interleaving distance by $\ourspace{n}$.

An important tool we need to state the next proposition is the smoothing $S_\varepsilon V$ of a persistence module $V$. For $\varepsilon\geq 0$ we define $S_\varepsilon V = \operatorname{im}(\eta_\varepsilon:V\to V[\varepsilon])$. We observe that $S_\varepsilon V\to V$ in the interleaving distance as $\varepsilon \to 0$. Since we work in the q-tame setting, the smoothing of a module by $\varepsilon >0$ will always be a pointwise finite-dimensional module, giving a simplification of a possibly quite complicated module. 

The next result gives a characterization of precompact sets in $\ourspace{n}$, as well as separability of $\mathbf{bg}_n$.
In particular, precompactness of a set $A \subseteq \ourspace{n}$ is equivalent to certain discretizations of $A$ being of finite representation type.

\begin{proposition}
  \label{proposition:separability-and-precompactness}
  Let $n \geq 1 \in \mathbb{N}$.
  \begin{itemize}
    \item A set $A \subseteq \ourspace{n}$ is precompact if and only if the set $\{S_\varepsilon V\vert_{\varepsilon\mathbb{Z}^n}:V\in A\}$ has finitely many isomorphism classes, for every $\varepsilon > 0$.
    \item The space $\mathbf{bg}_n$ is separable if and only if either the field $\mathbb{k}$ is countable or $n=1$.
  \end{itemize}
\end{proposition}

The next result identifies two common settings in which the collection of multiparameter persistence modules of interest is automatically precompact.

\begin{theorem}
  \label{theorem:examples-of-precompact-sets}
  Let $n \geq 1 \in \mathbb{N}$ and let $k \geq 0 \in \mathbb{N}$.
  \begin{itemize}
    \item Let $X$ be a finitely triangulable space.
    If $F \subseteq C(X ; \mathbb{R}^n)$ is uniformly bounded and equicontinuous, then $H_k(F) \subseteq \ourspace{n}$ is precompact.
    \item Let $Y$ be a compact metric space, let $S$ be the set of finite samples of $Y$, each one endowed with the normalized counting measure.
    Let $G$ denote either the Degree-Rips or Subdivision-Rips constructions postcomposed with $H_k$.
    Then $G(S) \subseteq \ourspace{n}$ is precompact.
  \end{itemize}
\end{theorem}

An important example for \cref{theorem:examples-of-precompact-sets} includes
any set of smooth maps $f : X \to \mathbb{R}^n$ from a compact Riemannian manifold $X$ to $\mathbb{R}^n$ with uniformly bounded $W^{2,\infty}$ norm.

\subsection{Related work}
The proof of (P1) in \cref{theorem:main-theorem} crucially relies on a novel construction (\cref{definition:qtame-to-pfd}) relating q-tame representations in one category to pointwise finite-dimensional representations in a different one (see proof of \cref{corollary:decomposition-qtLsc}).
This allows us to leverage the decomposition result for pdf modules \cite{botnan19} even in the q-tame case (as long as we work in the observable category).

An analogous result to (P2) was obtained independently by Petit, Schapira, and Waas~\cite{petit2021}, who establish that the interleaving distance reflects isomorphisms for a class of constructible sheaves.
Their approach uses a derived equivalence between $\gamma$-sheaves and sheaves for the standard  topology satisfying certain microsupport conditions (see~\cite{berkouk-petit}).
However, their setting requires \emph{constructibility} (the existence of a finite stratification on which the sheaf is locally constant), which is a significantly more restrictive condition than q-tameness. 
Another related result by Guillermou and Viterbo~\cite{guillermou23} shows reflection of isomorphisms by a metric of interest in microlocal sheaf theory (under certain constructibility assumptions).
To prove (P2) we follow a more direct approach that does not rely on microlocal sheaf theory and derived categories while admitting significantly more general conditions.
As shown in \cite{chazal}, it is impossible to relax these conditions in any reasonable way without losing (P2).
Our proof strategy was developed independently, and it is outlined in the last author's PhD thesis \cite[Section~4.5.1]{scoccola}, made public before \cite{petit2021,guillermou23}.

The proof of (P3) follows a general argument relying on categorical limits, which is used in \cite{bubenik-vergili} in the case of one-parameter persistence modules, and in \cite{cruz} in a different context where (P1) and (P2) are not satisfied.
The novelty here is in finding a suitable setup in which the argument can be carried out, while still satisfying (P1) and (P2).

Regarding \cref{proposition:separability-and-precompactness},
a similar but different result was announced by Cruz in a talk \cite{cruz-2}.
The difference with our result is that Cruz' applies to a category that is not equivalent to the observable category of q-tame modules, and which, for example, does not satisfy properties (P1) and (P2) in \cref{theorem:main-theorem}.

\section{Preliminaries}

We write $\mathbb{R}^n$ for the real $n$-space considered as a topological space or vector space and $\mathbf{R}^n$ if we consider it as a partial order, obtained as the $n$-fold product of $(\mathbf{R},\leq)$.
We write $\leq$ for the usual ordering relation on the product and $\ll$ for the relation of being strictly less in every component.
Throughout the paper we fix a base field $\mathbb{k}$ and the corresponding categories of (finite-dimensional) vector spaces $\mathbf{Vect}$ (respectively $\mathbf{vect}$).
An $n$-parameter \emph{persistence module} is a functor \[V:\mathbf{R}^n\to \mathbf{Vect}.\]
In other words, $V$ is a $\mathbb{k}$-representation of the poset $\mathbf{R}^n$.
More generally, a persistence module can be defined over arbitrary posets.
For $s\leq t$ in $\mathbf{R}^n$ we denote the map $V_s\to V_t$ by $V_{st}$.
An $n$-parameter persistence set is a functor $S:\mathbf{R}^n\to\mathbf{Set}$.

We say that a persistence module is \emph{pointwise finite-dimensional} if every $V_t$ is finite-dimensional.
An $n$-parameter persistence module $V$ is \emph{q-tame} if all maps $V_{s-\varepsilon,s}$ have finite rank.
Then by cofinality of the diagonal all $V_{st}$ with $s\ll t$ have finite rank.
If $\mathbf{C}$ is a category of persistence modules then we write $\mathbf{pfdC}$ or $\mathbf{qtC}$ for the full subcategory of pointwise finite-dimensional or q-tame modules, respectively. 

If $s\in\mathbb{R}^n,\varepsilon\in\mathbb{R}$ then $s+\varepsilon$ denotes the vector obtained by adding $\varepsilon$ to each component of $s$.
If $V$ is an $n$-parameter persistence module then $V[\varepsilon]$ defined by $V[\varepsilon]_s = V_{s+\varepsilon}$ with structure maps shifted accordingly is the shift of $V$ by $\varepsilon$.
Fix $\varepsilon>0$.
Then the structure maps $V_s\to V_{s+\varepsilon}$ define a morphism of persistence modules $\eta_\varepsilon^V$.
If $V$ is clear from the context then we allow ourselves to omit it.
Let $V,W$ be $n$-parameter persistence modules and let $\varepsilon\geq 0$.
A $\varepsilon$-\emph{interleaving} between $V$ and $W$ consists of morphisms \[f:V\to W[\varepsilon]~~~~\textrm{and}~~~~g:W\to V[\varepsilon]\] so that the diagrams

\[\begin{tikzcd}
	V && {V[2\varepsilon]} \\
	& {W[\varepsilon]}
	\arrow["{\eta_{2\varepsilon}}", from=1-1, to=1-3]
	\arrow["f"', from=1-1, to=2-2]
	\arrow["{g[\varepsilon]}"', from=2-2, to=1-3]
\end{tikzcd}\] and 
\[\begin{tikzcd}
	W && {W[2\varepsilon]} \\
	& {V[\varepsilon]}
	\arrow["{\eta_{2\varepsilon}}", from=1-1, to=1-3]
	\arrow["g"', from=1-1, to=2-2]
	\arrow["{f[\varepsilon]}"', from=2-2, to=1-3]
\end{tikzcd}\] commute.
We say that $V,W$ are $\varepsilon$-\emph{interleaved}.
Observe that a $0$-interleaving is an isomorphism.
We then let \[d_I(V,W) = \inf\{\varepsilon\geq 0:V,W\textrm{ are }\varepsilon\textrm{-interleaved}\}\] which we call the \emph{interleaving distance} of $V$ and $W$.
Observe that $d_I(V,W) = \varepsilon$ does not necessarily imply that $V$ and $W$ are $\varepsilon$-interleaved.
In particular $d_I(V,W) = 0$ does not give us an isomorphism $V\cong W$.
It is not difficult to show that $d_I$ is symmetric and that it satisfies the triangle inequality.
Thus $d_I$ is a pseudometric with values in $[0,\infty]$. 

Let $\mathbf{P}$ be a poset and let $s\in\mathbf{P}$.
We define the module $\mathsf{P}_s$ to be $\mathbb{k}$ on the upset of $s$ and $0$ elsewhere.
As morphisms we insert identities wherever possible and take all other structure maps to be zero.
A module $V$ over $\mathbf{P}$ is \emph{finitely presentable} if it can be written as \[\operatorname{cok}\left(\bigoplus_{j\in J}\mathsf{P}_j\to\bigoplus_{i\in I}\mathsf{P}_i\right)\] with $I,J$ finite multisets of elements of $\mathbf{P}$.

If $\mathbf{Q}\subseteq\mathbf{P}$ is an inclusion of a subposet and if $V$ is a module over $\mathbf{P}$ then we get a restricted module $V\vert_\mathbf{Q}$ by precomposition with the inclusion of $\mathbf{Q}$ in $\mathbf{P}$. 

A finite grid in $\mathbf{R}^n$ is a product $\{\varepsilon_1<\dots <\varepsilon_k\}^n$ with $\varepsilon_1,\dots,\varepsilon_k\in\mathbf{R}$. A countable grid is the $n$-fold product of sequences $\varepsilon_k\in\mathbf{R},k\in\mathbb{Z}$ without accumulation points.
A countable grid is regular if it is of the form $(\varepsilon\mathbf{Z})^n$ for some $\varepsilon>0$.

Let $\mathbf{P}\subseteq\mathbf{R}^n$ be a countable or finite grid and let $V$ be a module over $\mathbf{P}$.
We define the extension $\widetilde V:\mathbf{R}^n\to\mathbf{Vect}$ by \[\widetilde V_s = \left\{\begin{tabular}{cc}$V(\sup\{p\in\mathbf{P}:p\leq s\})$ & there is some $p\in\mathbf{P}$ with $p\leq s$ \\ $0$ & otherwise\end{tabular}\right.\] where the structure maps come from those of $V$ where we use the fact that $\sup \{p\in\mathbf{P}:p\leq s\}\leq \sup\{p\in\mathbf{P}:p\leq t\}$ whenever $s\leq t$ and when there is some $p\in\mathbf{P}$ so that $p\leq s$.
We call a module of this form a $\mathbf{P}$-extension; we then also say that it is defined on $\mathbf{P}$.
Let $V$ be an $n$-parameter module and let $\mathbf{P}\subseteq \mathbf{R}^n$ be a countable or finite grid.
We denote by $V_\mathbf{P}$ the module obtained by first restricting to $\mathbf{P}$ and then extending.
We call this procedure a \emph{restriction-extension}.
This construction is functorial.
We state two lemmas on grid extensions. 

\begin{lemma}
  Let $V$ be a pointwise finite-dimensional $n$-parameter module.
  Then $V$ is finitely presentable if and only if there is a finite grid $\mathbf{P}\subseteq\mathbf{R}^n$ so that $V \cong V_\mathbf{P}$.
\end{lemma}

\begin{lemma}
  Let $\mathbf{P}$ be a finite or countable grid in $\mathbf{R}^n$ and let $V$ be defined on $\mathbf{P}$.
  Let $s<t$ so that every grid point $p$ with $p\leq t$ also satisfies $p\leq s$. Then $V_{s,t}$ is an isomorphism. 
\end{lemma}

A proof of the first lemma can be found in \cite{bauer23}.

\subsection{Constructing the Observable Category}

The most important tool we need is the observable category of $n$-parameter persistence modules. It consists of all persistence modules, but identifies features not seen by the interleaving distance with zero. To achieve this, one localizes the ephemeral modules, which are precisely those modules at distance zero to the zero module. 
We follow closely and generalize the approach of \cite{chazal}, which treats the case of one-parameter modules. A more general account of the below results can be found in \cite{harsu2024}, which uses the language of Scott sheaves. In contrast, we perform all constructions explicitly in the context of modules over $\mathbf{R}^n$. 

We begin with the concept of localization to be used for defining the observable category of multiparameter persistence modules: 

\begin{definition}
  Let $\mathbf{A}$ be an Abelian category.
  A full subcategory $\mathbf{E}\subseteq\mathbf{A}$ is a \emph{Serre subcategory} if for every short exact sequence 

\[\begin{tikzcd}
	0 & {V'} & V & {V''} & 0
	\arrow[from=1-1, to=1-2]
	\arrow[from=1-2, to=1-3]
	\arrow[from=1-3, to=1-4]
	\arrow[from=1-4, to=1-5]
\end{tikzcd}\] in $\mathbf{A}$ we have the equivalence of the following statements:

  \begin{enumerate}
    \item $V',V''\in\mathbf{E}$.
    \item $V\in\mathbf{E}$.
  \end{enumerate}

  A morphism $\varphi:V\to W$ in $\mathbf{A}$ is a \emph{weak equivalence} if $\ker \varphi$ and $\operatorname{cok} \varphi$ are in $\mathbf{E}$.
  The above condition ensures that the composition of weak equivalence remains a weak equivalence: If $\varphi:U\to V,\psi:V\to W$ are weak equivalence then we have an exact sequence  
  \[\begin{tikzcd}
    0 & {\ker\varphi} & {\ker\psi\circ\varphi} & {\ker\psi} & {\operatorname{cok}\varphi} & {\operatorname{cok} \psi\circ\varphi} & {\operatorname{cok}\psi} & 0
    \arrow[from=1-1, to=1-2]
    \arrow[from=1-2, to=1-3]
    \arrow[from=1-3, to=1-4]
    \arrow[from=1-4, to=1-5]
    \arrow[from=1-5, to=1-6]
    \arrow[from=1-6, to=1-7]
    \arrow[from=1-7, to=1-8]
  \end{tikzcd}\] which shows that $\psi\circ\varphi$ is also a weak equivalence.
  
  The \emph{Serre localization} $\pi:\mathbf{A} \to \mathbf{A}/\mathbf{E}$ is a functor carrying weak equivalence to isomorphisms with the universal property that any other such functor factors over $\pi$.
\end{definition}

\begin{definition}
  We call an $n$-parameter module $V$ \emph{ephemeral} if for all $s\in \mathbf{R}^n$ and all $\varepsilon>0$ the map $V_{s-\varepsilon,s}$ vanishes.
  We write $\mathbf{Eph}_n$ for the category of ephemeral modules. 
\end{definition}

\begin{lemma}
  $\mathbf{Eph}_n\subseteq \mathbf{Pers}_n$ is a Serre subcategory.

  \begin{proof}
      We consider a short exact sequence   
      \[\begin{tikzcd}
        0 & {V'} & V & {V''} & 0
        \arrow[from=1-1, to=1-2]
        \arrow[from=1-2, to=1-3,"\varphi"]
        \arrow[from=1-3, to=1-4,"\psi"]
        \arrow[from=1-4, to=1-5]
      \end{tikzcd}\] in $\mathbf{Pers}_n$.
      If $V$ is ephemeral then $V'$ and $V''$ must be, too, since they are submodule and quotient respectively.
      Conversely, let $V'$ and $V''$ be ephemeral.
      Let $\varepsilon >0$.
      Choose some $0<\varepsilon'<\varepsilon$ and consider the diagram 
      \[\begin{tikzcd}
        & {V_{s}'} & {V_{s}} \\
        0 & {V_{s-\varepsilon'}'} & {V_{s-\varepsilon'}} & {V_{s-\varepsilon'}''} & 0 \\
        && {V_{s-\varepsilon}} & {V_{s-\varepsilon}''}
        \arrow["{\varphi_{s}}", from=1-2, to=1-3]
        \arrow[from=2-1, to=2-2]
        \arrow[from=2-2, to=1-2]
        \arrow["{\varphi_{s-\varepsilon'}}", from=2-2, to=2-3]
        \arrow[from=2-3, to=1-3]
        \arrow["{\psi_{s-\varepsilon'}}", from=2-3, to=2-4]
        \arrow[from=2-4, to=2-5]
        \arrow["\alpha", dotted, from=3-3, to=2-2]
        \arrow[from=3-3, to=2-3]
        \arrow["{\psi_{s-\varepsilon}}", from=3-3, to=3-4]
        \arrow[from=3-4, to=2-4]
      \end{tikzcd}\] Since $\psi_{s-\varepsilon'}\circ V_{s-\varepsilon',s-\varepsilon}  = V''_{s-\varepsilon',s-\varepsilon}\circ \psi_{s-\varepsilon} = 0\circ\psi_s = 0$ the dotted arrow exists so that everything commutes.
      Thus $V_{s-\varepsilon,s}$ factors over $V_{s-\varepsilon',s}' = 0$ and we are done.
  \end{proof}
\end{lemma}

Our goal is now to localize $\mathbf{Pers}_n$ by $\mathbf{Eph}_n$.
We first give an explicit definition of the observable category and we will show that it satisfies the universal property of $\mathbf{Pers}_n/\mathbf{Eph}_n$ at the end of this chapter.

\begin{definition}
  We define the \emph{observable category} $\mathbf{Ob}_n$ as follows: The objects are the $n$-parameter persistence modules.
  A morphism $\varphi:V\to W$ in $\mathbf{Ob}_n$ consists of a family of maps $\varphi_{s,s+\varepsilon}:V_s\to W_{s+\varepsilon}$ for $s\in\mathbf{R}^n,\varepsilon>0$ so that 

\[\begin{tikzcd}
	{V_s} & {W_{s+\varepsilon}} \\
	{V_{s+\varepsilon''}} & {W_{s+\varepsilon'}}
	\arrow["{\varphi_{s,s+\varepsilon}}", from=1-1, to=1-2]
	\arrow[from=1-1, to=2-1]
	\arrow["{\varphi_{s+\varepsilon'',s+\varepsilon'}}"', from=2-1, to=2-2]
	\arrow[from=2-2, to=1-2]
\end{tikzcd}\] commutes for all $0\leq\varepsilon''<\varepsilon'\leq \varepsilon$.
We also call these observable morphisms.
For $\varphi:U\to V$ and $\psi:V\to W$ in $\mathbf{Ob}_n$ we define \[(\psi\circ\varphi)_{s,s+\varepsilon} = \psi_{s+\varepsilon',s+\varepsilon}\circ\varphi_{s,s+\varepsilon'}\] for some $0<\varepsilon'<\varepsilon$.
By the above this is independent of the choice of intermediate scalar.
The identity on a module $V$ is given by the structure maps: $\operatorname{id}_V = (V_{s,s+\varepsilon})_{s\in\mathbf{R}^n,\varepsilon >0}$.
Associativity is clear. 
\end{definition}

Observe that if $\varphi:V\to W$ is a morphism of persistence modules in the usual sense then we can interpret it as an observable morphism by $\varphi_{s,s+\varepsilon} = W_{s,s+\varepsilon}\circ \varphi_s$.
This defines a functor \[\pi:\mathbf{Pers}_n\to \mathbf{Ob}_n.\] 
We will later see that $(\mathbf{Ob}_n,\pi)$ satisfy the universal property of Serre localization.

\begin{remark}
  Being q-tame is an observable property in the following sense: If $V,W$ are modules with $V\cong W$ in the observable category so that $V$ is q-tame then $W$ is also q-tame.
  Indeed, If we have an observable isomorphism $\varphi:W\to V$ then we can write a structure map $W_{s-\delta,s}$ as 
  \[\begin{tikzcd}
    {W_{s-\delta}} & {V_{s-\delta'}} & {V_{s-\delta''}} & {W_s}
    \arrow["{\varphi_{s-\delta,s-\delta'}}", from=1-1, to=1-2]
    \arrow[from=1-2, to=1-3]
    \arrow["{\varphi_{s-\delta',s}^{-1}}", from=1-3, to=1-4]
  \end{tikzcd}\] where the middle has finite image.
  On the other hand, being pointwise finite-dimensional is not an observable property: If $E$ is a nonzero ephemeral module then $\bigoplus^\infty_{i=1}E$ is no longer pointwise finite-dimensional.
  It is nonetheless isomorphic to $0$ in the observable category. 
\end{remark}

\subsection{Semicontinuity}

The definitions from this chapter are generalized from those in \cite{Schmahl22}.

\begin{definition}
  An $n$-parameter module $V$ is \emph{lower semicontinuous} (respectively \emph{upper semicontinuous}) if the canonical map \[\underset{s-\varepsilon}{\operatorname{colim}}~V_{s-\varepsilon}\to V_s\] respectively \[V_s\to \lim_{s+\varepsilon} V_{s+\varepsilon}\] is an isomorphism for all $s\in\mathbf{R}^n$.
  We write $\mathbf{Lsc}_n$ for the category of lower semicontinuous modules and $\mathbf{Usc}_n$ for the category of upper semicontinuous modules. 
\end{definition}

\begin{remark}
  An equivalent way of writing the lower semicontinuity condition is \[V_s = \bigcup_{\varepsilon>0}\operatorname{im} V_{s-\varepsilon,s}\] for all $s\in\mathbf{R}^n$.
\end{remark}

\begin{corollary}
  A module $V$ is lower semicontinuous if and only if $V \cong \underset{\varepsilon >0}{\operatorname{colim}}~ V[-\varepsilon]$.
  Dually, $V$ is upper semicontinuous if and only if  $V \cong \underset{\varepsilon>0}{\lim}~V[\varepsilon]$.
\end{corollary}

\begin{definition}
  Let $V$ be a $n$-parameter persistence module.
  We construct a module $\overline{V}$ by taking \[\overline V_s = \underset{\varepsilon>0}{\operatorname{colim}}~V_{s-\varepsilon}\] with the structure maps defined by the appropriate colimits.
  By the functoriality of the colimit this indeed defines a persistence module, which is also called the \emph{lower envelope} of $V$.
  Dually, taking \[\underline V_s = \lim_{\varepsilon>0}V_{s+\varepsilon}\] defines a module $\underline{V}$, called the \emph{upper envelope} of $V$.
  From the universal property of the colimit we get a map $\xi^V:\overline V\to V$ assembled from the structure maps. There is a dual map $\zeta^V:V\to \underline V$.
\end{definition}

Let $\varphi:V\to W$ be an observable morphism of $n$-parameter persistence modules.
Fix some $\delta>0,s\in\mathbf{R}^n$ and let $\varepsilon>\delta>0$.
We then have maps \[V_{s-\varepsilon}\to W_{s-\delta}\to\overline W_s\] which are independent of the choice of $\delta$.
We let $\overline{\varphi}_s:\overline V_s\to\overline W_s$ be the map obtained from the universal property of the colimit when applied to the above data.
Again by the functoriality of the colimit this gives us a functor \[\overline{\,\cdot\,}:\mathbf{Ob}_n\to\mathbf{Pers}_n\] and the dual construction yields a functor \[\underline{\,\cdot\,}:\mathbf{Ob}_n\to\mathbf{Pers}_n.\]  We call these the \emph{lower} and \emph{upper envelope functors}, respectively.

\begin{lemma}
  Let $V$ be q-tame.
  Then $\overline V$ is also q-tame.

  \begin{proof}
    Let $s\ll t$ and choose $\delta>0$ so that $s\ll t-\delta$.
    We can then  factor \[\overline V_{s,t} = \overline V_s\to V_s \to V_{t-\delta}\to\overline V_t\] which shows the claim since $V_{s}\to V_{t-\delta}$ has finite image. 
  \end{proof}
\end{lemma}

\subsection{A Three-Fold Equivalence}

The goal of this section is to establish the equivalences \[\mathbf{qtLsc}_n\simeq \mathbf{qtOb}_n\simeq \mathbf{qtUsc}_n.\] 
To this end we show that we have adjunctions \[\overline{\,\cdot\,}\dashv \pi\dashv \underline{\,\cdot\,}\] and that these are adjoint equivalences when restricted to the appropriate subcategories.
The existence of these equivalences and of the adjunctions we use is already remarked in \cite{chazal}.
This chain of equivalences will give us the flexibility to always choose a suitable incarnation of the observable category for each purpose.

\begin{proposition}
  We have a chain of adjunctions \[\overline{\,\cdot\,}\dashv \pi\dashv \underline{\,\cdot\,}.\]

  \begin{proof}
    We show the adjunction $\overline{\,\cdot\,}\dashv \pi$, the other part is entirely dual.
    We have to establish natural bijections \[\operatorname{Hom}(\overline V,W) \cong  \operatorname{Hom}(V,\pi (W))\] for $V\in\mathbf{Ob}_n$ and $W\in \mathbf{Pers}_n$.
    We first construct the map \[(\cdot)^\flat:\operatorname{Hom}(\overline V,W)\to\operatorname{Hom}(V,\pi(W)).\] 
    If $\varphi:\overline{V}\to W$ we define $\varphi^\flat:V\to \pi(W)$ as $\varphi^\flat_{s-\varepsilon,s} = \varphi_s\circ \operatorname{inc}_{s-\varepsilon,s}$ where $\operatorname{inc}_{s-\varepsilon}:V_{s-\varepsilon,s}\to\overline V_s$ is the inclusion into the colimit. 
    Let $\varepsilon\geq \varepsilon'>\varepsilon'' \geq 0$ and consider the diagram 
    \[\begin{tikzcd}[column sep=large]
      {V_{s-\varepsilon}} & {\overline V_{s}} & {W_s} \\
      {V_{s-\varepsilon'}} & {\overline V_{s-\varepsilon''}} & {W_{s-\varepsilon''}}
      \arrow["{\operatorname{inc}_{s-\varepsilon,s}}", from=1-1, to=1-2]
      \arrow[from=1-1, to=2-1]
      \arrow["{\varphi_s}", from=1-2, to=1-3]
      \arrow["{\operatorname{inc}_{s-\varepsilon',s-\varepsilon''}}"', from=2-1, to=2-2]
      \arrow[from=2-2, to=1-2]
      \arrow["{\varphi_{s-\varepsilon''}}"', from=2-2, to=2-3]
      \arrow[from=2-3, to=1-3]
    \end{tikzcd}\] The left square commutes by construction and the right square commutes since $\varphi$ is already a morphism of persistence modules.
    Therefore $\varphi^\flat$ is indeed an observable morphism.
    We now construct the map \[\operatorname{Hom}(\overline V,W) \leftarrow \operatorname{Hom}(V,\pi (W)):(\cdot)^\sharp.\]
    Let $\psi:V\to \pi(W)$ be an observable morphism.
    For all $s\leq t$ and all $\varepsilon> 0$ the diagram 
    \[\begin{tikzcd}
      {V[-\varepsilon]_s} & {W_s} \\
      {V[-\varepsilon]_t} & {W_t}
      \arrow["{\psi_{s-\varepsilon,s}}", from=1-1, to=1-2]
      \arrow[from=1-1, to=2-1]
      \arrow["{\psi_{s-\varepsilon,t}}", from=1-1, to=2-2]
      \arrow[from=1-2, to=2-2]
      \arrow["{\psi_{t-\varepsilon,t}}"', from=2-1, to=2-2]
    \end{tikzcd}\] commutes.
    We thus have morphisms \[\psi^\varepsilon:V[-\varepsilon]\to W\] and since \[\overline V = \underset{\varepsilon >0}{\operatorname{colim}}~V[-\varepsilon]\] we can take \[\psi^\sharp = \underset{\varepsilon >0}{\operatorname{colim}}~\psi^\varepsilon\] which is a morphism of persistence modules by construction.
    We now verify that these constructions are indeed inverses. Let $\varphi:\overline V\to W$ be a morphism of persistence modules.
    Then \[\varphi^{\flat\sharp}_s = \underset{\varepsilon >0}{\operatorname{colim}}~\varphi^{\sharp}_{s-\varepsilon,s} =\underset{\varepsilon >0}{\operatorname{colim}} (\varphi_s\circ\operatorname{inc}_{s-\varepsilon,s}) = \varphi_s\] since a map out of $\overline V_s$ is precisely determined by the values it takes on the constituents of the colimit.
    Conversely, let $\psi:V\to \pi(W)$ be an observable morphism.
    We can then compute \[\psi^{\sharp\flat}_{s-\varepsilon,s} = \psi_s^\sharp\circ\operatorname{inc}_{s-\varepsilon,s} = \underset{\delta >0}{\operatorname{colim}} (\psi_{s-\delta,s})\circ\operatorname{inc}_{s-\varepsilon,s} = \psi_{s-\varepsilon,s}\] which shows that $\sharp$ and $\flat$ are inverses.
    Naturality follows since all constructions were canonical. 
  \end{proof}
\end{proposition}

\begin{proposition}\label{prop:adjoint_equivalence}
  The above adjunctions become adjoint equivalences  \[\mathbf{Lsc}_n \simeq \mathbf{Ob}_n\simeq \mathbf{Usc}_n\] when appropriately restricted.

  \begin{proof}
    We again focus on the $\overline{\,\cdot\,}\dashv \pi$-case.
    To show that we have an adjoint equivalence it suffices to show that the unit \[\eta_V:V\to \pi(\overline V)\] and the counit \[\varepsilon_W:\overline{\pi(W)}\to W\] are isomorphisms for all $V\in\mathbf{Ob}_n$ and $W\in\mathbf{Lsc}_n$.
    The unit arises via \[(\cdot)^\sharp:\operatorname{Hom}(\overline V,\overline V)\to \operatorname{Hom}(V,\pi(\overline{V}))\] \[\operatorname{id}_{\overline V}\mapsto \eta_V\] and the counit is obtained by \[(\cdot)^\sharp:\operatorname{Hom}(\pi(W),\pi(W))\to \operatorname{Hom}(\overline{\pi(W)},W)\] \[\operatorname{id}_{\pi(W)}\mapsto \varepsilon_W\] Let $\varepsilon,\delta>0$.
    Observe that we have $\eta_{V,s-\varepsilon,s} = \operatorname{inc}_{s-\varepsilon,s}$ and recall that we have a canonical map $\xi^V:\overline V\to V$ in the other direction.
    The map $\xi^V_{s,s+\delta}$ is glued from structure maps $V_{s-\delta'} \to V_{s}\to V_{s+\delta}$.
    Hence the composition 

\[\begin{tikzcd}
	{V_{s-\varepsilon}} & {\overline V_s} & {V_{s+\delta}}
	\arrow["{\operatorname{inc}_{s-\varepsilon,s}}", from=1-1, to=1-2]
	\arrow["{\xi^V_{s,s+\delta}}", from=1-2, to=1-3]
\end{tikzcd}\] is just the structure map $V_{s-\varepsilon}\to V_{s+\delta}$.
Conversely consider the composition 
\[\begin{tikzcd}
	{\overline{V}_{s-\varepsilon}} & {V_s} & {\overline V_{s+\delta}}
	\arrow["{\xi^V_{s-\varepsilon,s}}", from=1-1, to=1-2]
	\arrow["{\operatorname{inc}_{s,s+\delta}}", from=1-2, to=1-3]
\end{tikzcd}\] which is a factorization of the restriction map $\overline V_{s-\varepsilon}\to \overline V_{s+\delta}$.
Therefore $\eta_V$ is an observable isomorphism with inverse $\xi^V$.
For the counit we observe that \[\varepsilon_{W,s} = \underset{\varepsilon>0}{\operatorname{colim}}~W_{s-\varepsilon,s}\] which defines an isomorphism by lower semicontinuity. 
  \end{proof}
\end{proposition}

\begin{corollary}
  We have a chain of equivalences \[\mathbf{qtLsc}_n \simeq \mathbf{qtOb}_n\simeq \mathbf{qtUsc}_n.\] 

  \begin{proof}
    We have already shown that $\overline{\,\cdot\,}$ (and dually, $\underline{\,\cdot\,}$) preserves q-tameness and that q-tameness is an observable property.
    Thus the equivalences follow immediately from \cref{prop:adjoint_equivalence}.
  \end{proof}
\end{corollary}

 \begin{remark}
  Interleavings are preserved when passing between these categories.
  Clearly, the $\pi$-functor preserves interleaving relations.
  Dually, we show that $\overline{\,\cdot\,}$ also preserves interleavings.
  Let $V,W\in\mathbf{Ob}_n$ and consider a $\varepsilon$-interleaving $f:V\to W[\varepsilon],g:W\to V[\varepsilon]$.
  By functoriality we get commutative diagrams 
  \[\begin{tikzcd}
    {\overline V} && {\overline{V[2\varepsilon]}} \\
    & {\overline{W[\varepsilon]}}
    \arrow["{\eta_{2\varepsilon}}", from=1-1, to=1-3]
    \arrow[from=1-1, to=2-2]
    \arrow[from=2-2, to=1-3]
  \end{tikzcd}\] and 
  \[\begin{tikzcd}
    {\overline W} && {\overline{W[2\varepsilon]}} \\
    & {\overline{V[\varepsilon]}}
    \arrow["{\eta_{2\varepsilon}}", from=1-1, to=1-3]
    \arrow[from=1-1, to=2-2]
    \arrow[from=2-2, to=1-3]
  \end{tikzcd}\] Hence it remains to observe that the functor $\overline{\,\cdot\,}$ commutes with shifts.
  Indeed, \[\overline{V[\varepsilon]}_s = \underset{\delta>0}{\operatorname{colim}}~V[\varepsilon]_{s-\delta} = \underset{\delta>0}{\operatorname{colim}}~V_{s+\varepsilon-\delta} = \overline{V}[\varepsilon]_s \] which shows the claim.
  Thus all metric properties carry over between these categories. 
 \end{remark}

 \begin{remark}
  The isomorphism classes of any of the above categories of q-tame modules form a set, in particular, we do not run into any size issues when treating the isomorphism classes as points of a space.
  We show this in the case of q-tame and lower semicontinuous modules.
  Let $V$ be such a module.
  Then at every index $s\in\mathbf{R}^n$ we can write \[V_s = \bigcup_{\varepsilon>0} \operatorname{im}(V_{s-\varepsilon}\to V_s)\] where we identify the images in $V_s$.
  Each image is finite-dimensional by q-tameness, hence $V_s$ is of countable dimension.
  It follows that every structure map $V_{s,t}$ is a morphism between vector spaces of countable dimension.
  Thus $V$ is determined uniquely by a countable-dimensional vector space for each $s\in\mathbb{R}^n$ and a morphism of such vector spaces for every pair $s\leq t$ in $\mathbb{R}^n$.
  It is now clear that all possible choices of such data, even those not forming persistence modules, are a set. 
 \end{remark}

 \begin{definition}
  We write $\ourspace{n}$ for the extended pseudometric space of isomorphism classes of the category $\mathbf{qtOb}_n$. By the above remark, this is indeed a set. 
 \end{definition}

 The next proposition states that $\mathbf{Ob}_n$ is indeed the Serre localization of $\mathbf{Pers}_n$ at $\mathbf{Eph}_n$.
  The proof of the 1-parameter case in \cite{chazal} generalizes verbatim to the $n$-parameter case.

\begin{proposition}
  The pair $(\mathbf{Ob}_n,\pi)$ satisfies the universal property of\/ $\mathbf{Pers}_n/\mathbf{Eph}_n$.
  More explicitly, any functor $F:\mathbf{Pers}_n\to \mathbf{C}$ taking weak equivalence to isomorphisms factors uniquely through~$\pi$.
\end{proposition}

\begin{corollary}
 The observable category is an Abelian category.
\end{corollary}

\begin{corollary}
  The observable category of q-tame modules is Abelian.
  \begin{proof}
    We first observe that $\mathbf{qtOb}_n$ is the Serre localization of the full subcategory of q-tame modules at the Serre subcategory of ephemeral modules.
    Thus it suffices to show that the full subcategory of q-tame persistence modules forms an Abelian category, which can readily be seen by verifying that it is closed under finite direct sums, kernels, and cokernels.
  \end{proof}
\end{corollary}

\section{Algebraic Constructions: Decomposition Results}\label{section:decomposition}

We now consider the decomposition theorems for q-tame and continuous modules obtained in \cite{Schmahl22} for the $1$-parameter case, which we generalize to $n$ parameters.
We first show that we have additive decomposition in the category of lower semicontinuous modules, which dually yields decomposition into products in the upper semicontinuous category.
These results combine to givethe first part of our main theorem, which asserts an essentially unique decomposition into biproducts in the observable category of q-tame persistence modules.

\subsection{Additive Decomposition for Lower Semicontinuous Modules}

In this section we study the additive decomposition of semicontinuous modules, which will allow us to establish the Krull--Schmidt property.
Our general strategy will be to transform a lower semicontinuous q-tame module to a pointwise finite-dimensional module over another poset by an equivalence of categories.
We will then be able to decompose this module by the well-known decomposition result for pointwise finite-dimensional modules in \cite{botnan19} and transfer the decomposition back to $\mathbf{R}^n$.

\begin{definition}
  \label{definition:qtame-to-pfd}
  We consider the sets \[\mathbf{S}^n = \{(s-\varepsilon,s):s\in\mathbf{R}^n,\varepsilon > 0\}\] and \[\overline{\mathbf{S}^n} = \{(s-\varepsilon,s):s\in\mathbf{R}^n,\varepsilon \geq 0\}\] with the partial order inherited as subsets of $\mathbf{R}^{2n}$.
\end{definition}

We have a functor
\[\operatorname{im}:\mathbf{Pers}_n\to \mathbf{Pers}(\overline{\mathbf{S}^n})\] given on objects by 
\[\operatorname{im}(V)_{(s-\varepsilon,s)} = \operatorname{im}(V_{s-\varepsilon,s}:V_{s-\varepsilon}\to V_s)\]
and on morphisms by restriction to images.
In the other direction, there is a functor \[\operatorname{dg}:\mathbf{Pers}(\overline{\mathbf{S}^n}) \to \mathbf{Pers}_n\] which restricts a module over $\mathbf{S}^n$ to the diagonal by precomposing with the inclusion $\mathbf{R}^n\to \overline{\mathbf{S}^n}$.

\begin{definition}
  We say that a module $V$ over $\mathbf{S}^n$ (respectively, $\overline{\mathbf{S}^n}$) has the \emph{mono-epi} property if for all $\varepsilon>\delta>0$ (respectively, $\varepsilon>\delta\geq 0$) all horizontal maps \[V_{(s-\varepsilon,s),(s-\delta,s)}:V_{(s-\varepsilon,s)}\to V_{(s-\delta,s)}\] are monomorphisms and all vertical maps \[V_{(s,s+\delta),(s,s+\varepsilon)}: V_{(s,s+\delta)}\to V_{(s,s+\varepsilon)}\] are epimorphisms.
\end{definition}

\begin{lemma}\label{lemma:equivalence_mono_epi}
  The category $\mathbf{Pers}_n$ is equivalent to the category of modules over $\overline{\mathbf{S}^n}$ that have the mono-epi property via the above functors. 

  \begin{proof}
    We can immediately verify that  $\operatorname{dg}\circ\operatorname{im}$ is isomorphic to the identity functor on $\mathbf{Pers}_n$.
    Conversely, to verify that $\operatorname{im}\circ \operatorname{dg}$ is isomorphic to the identity on the category of modules over $\overline{\mathbf{S}^n}$ that have the mono-epi property, the factorization 
        \[\begin{tikzcd}
            {W_{(s-\delta,s)}} & {W_{(s,s)}} \\
            {W_{(s-\delta,s-\delta)}}
            \arrow[hook, from=1-1, to=1-2]
            \arrow[two heads, from=2-1, to=1-1]
        \end{tikzcd}\]
    of the structure map $W_{(s-\delta,s-\delta)}\to W_{(s,s)}$ for $\delta > 0$ gives \(W_{s-\delta,s} \cong \operatorname{im}(W_{(s-\delta,s-\delta)}\to W_{(s,s)})\).
  \end{proof}
\end{lemma}

Moreover, we have a restriction functor \[\operatorname{res}:\mathbf{Pers}(\overline{\mathbf{S}^n})\to\mathbf{Pers}(\mathbf{S}^n)\] which forgets the values of a persistence module on the diagonal.

In the other direction, we define a functor \[\operatorname{ext}^L:\mathbf{Pers}(\mathbf{S}^n)\to \mathbf{Pers}(\overline{\mathbf{S}^n})\] which extends a module $V$ to the diagonal via a colimit from below as
\[\operatorname{ext}^L(V)_{(s,s)} = \underset{\varepsilon>0}{\operatorname{colim}}~V_{(s-\varepsilon,s)},\]
and which defines structure maps and morphisms via the universal property of the colimit: 
If $\varphi:V\to W$ is a morphism of modules over $\mathbf{S}^n$, then the $\varphi_{(s-\varepsilon,s)}:V_{(s-\varepsilon,s)}\to W_{(s-\varepsilon,s)}$ define a natural transformation of colimit data, and so the universal property of the colimit gives us a morphism $\operatorname{ext}^L(V)_{(s,s)}\to \operatorname{ext}^L(W)_{(s,s)}$.

\begin{remark}
  By instead taking a limit from above, we can define another functorial extension to the diagonal
  \[\operatorname{ext}^R:\mathbf{Pers}(\mathbf{S}^n)\to \mathbf{Pers}(\overline{\mathbf{S}^n}).\]
  In fact, we have adjunctions
  \[\operatorname{ext}^L \dashv \operatorname{res}\dashv \operatorname{ext}^R.\]
  We will, however, not make use of this fact. 
\end{remark}

\begin{definition}
  We say that a module $V$ over $\mathbf{S}^n$ is \emph{left continuous} if for all $(r,s)\in\mathbf{S}^n$ we have \[V_{r,s} = \underset{\delta >0}{\operatorname{colim}}~ V_{r-\delta,s}.\]
\end{definition}

\begin{lemma}\label{lemma:mono-epi-left-continuous}
  For a module $V$ over $\mathbf{S}^n$, the following are equivalent:

  \begin{enumerate}
    \item $V$ is left continuous and has the mono-epi property.
    \item $\operatorname{ext}^L(V)$ has the mono-epi property.
  \end{enumerate}

  \begin{proof}
    Let $\operatorname{ext}^L(V)$ have the mono-epi property.
    Then $V$ also has the mono-epi property.
    Moreover, let $s\in\mathbf{R}^n, \varepsilon>0$ and $t = s+\varepsilon$.
    We then have a surjection \[\underset{\delta >0}{\operatorname{colim}}~V_{(s-\delta,s)}\twoheadrightarrow V_{s,t}\] which factors as 
    \[\begin{tikzcd}
      {\underset{\delta>0}{\operatorname{colim}}~V_{(s-\delta,s)}} & {\underset{\delta>0}{\operatorname{colim}}~V_{(s-\delta,t)}} & {V_{(s,t)}}
      \arrow[from=1-1, to=1-2]
      \arrow[two heads, from=1-2, to=1-3]
    \end{tikzcd}\] by the universal property of the colimit.
    This implies that the left map is an isomorphism, too.
    Conversely, the implication follows from the fact that the colimit functor preserves epimorphisms. 
  \end{proof}
\end{lemma}

\begin{proposition}\label{proposition:equivalence-lsc-mono-epi}
  The full subcategory $\mathbf{Lsc}_n$ of lower semicontinuous modules is equivalent to the full subcategory of left continuous mono-epi modules over $\mathbf{S}^n$.

  \begin{proof}
    The category $\mathbf{Lsc}_n$ is clearly equivalent to the category of modules over $\overline{\mathbf{S}^n}$ that have the mono-epi property and restrict to a lower semicontinuous module on the diagonal, by an application of \cref{lemma:equivalence_mono_epi}.
    We show that the latter is equivalent to the category of left continuous and mono-epi modules over $\mathbf{S}^n$ via the extension and restriction functors, proving the claim.

    Let $W$ be a left continuous and mono-epi module over $\mathbf{S}^n$.
    We first show that $V = \operatorname{ext}^L(W)$ has a lower semicontinuous module $s\mapsto V_{(s,s)}$ on the diagonal.
    Let $s\in\mathbf{R}^n$.
    Since colimits commute with colimits, we can compute \[\underset{\varepsilon >0}{\operatorname{colim}}~V_{(s-\varepsilon,s-\varepsilon)} = \underset{\varepsilon >0}{\operatorname{colim}}~\underset{\delta >0}{\operatorname{colim}}~W_{(s-\varepsilon-\delta,s-\varepsilon)} = \underset{\delta >0}{\operatorname{colim}}~\underset{\varepsilon >0}{\operatorname{colim}}~ W_{(s-\varepsilon-\delta,s-\varepsilon)} \cong \underset{\delta >0}{\operatorname{colim}}~W_{(s-\delta,s)} \cong V_{(s,s)}\] which proves lower semicontinuity.
    This shows that the extension functor maps into the desired category.
    First applying $\operatorname{ext}^L$ and then restricting is clearly isomorphic to the identity.
    Conversely, let $V$ be a mono-epi module with a lower semicontinuous module on the diagonal.
    Dually, by the mono-epi property, the modules $V_{(s-\delta,s)}$ are the image of $V_{(s-\delta,s-\delta)}$ in $V_{(s,s)}$.
    By lower semicontinuity, we have
    \[V_{(s,s)} \cong \underset{\varepsilon >0}{\operatorname{colim}}~V_{(s-\varepsilon,s-\varepsilon)} \cong \underset{\varepsilon >0}{\operatorname{colim}}~V_{(s-\varepsilon,s)},\] 
    which implies that $\operatorname{ext}^L(\operatorname{res})(V)$ is isomorphic to $V$ via the universal property of the colimit.
    This is clearly natural in $V$. 
  \end{proof}
\end{proposition}

\begin{corollary}
  \label{corollary:decomposition-qtLsc}
  The category $\mathbf{qtLsc}_n$ has essentially unique additive decompositions into indecomposables with local endomorphism rings.

  \begin{proof}
By the previous proposition, the category $\mathbf{qtLsc}_n$ is equivalent to the category of pointwise finite-dimensional left continuous mono-epi modules over $\mathbf{S}^n$.
The category of pointwise finite-dimensional modules over any poset has essentially unique decompositions into indecomposables with local endomorphism rings.
Hence we immediately see that all endomorphism rings of the indecomposables in the category of left continuous mono-epi modules over $\mathbf{S}^n$ are local.
This implies that any decomposition into indecomposables is essentially unique.

To obtain a decomposition, let $V\in\mathbf{qtLsc}$ and let $W = (\operatorname{res}\circ \operatorname{im})(V)$.
By \cite{botnan19}, we can decompose as
\[W \cong \bigoplus_{i\in I}W_i,\]
where the $W_i$ are pointwise finite-dimensional over $\mathbf{S}^n$ and have local endomorphism rings.
It remains to verify that the $W_i$ are left continuous and mono-epi.
The functor $\operatorname{ext}^L$ is defined by a colimit on the diagonal and the identity everywhere else.
We know that $\operatorname{ext}^L(W)$ has the mono-epi property by \cref{proposition:equivalence-lsc-mono-epi}, and we have \[\operatorname{ext}^L(W)  \cong \operatorname{ext}^L\left(\bigoplus_{i\in I}W_i\right) \cong \bigoplus_{i\in I} \operatorname{ext}^L(W_i)\] since colimits commute with colimits.
Since they sum to a mono-epi module, it follows that the $\operatorname{ext}^L(W_i)$ have the mono-epi property.
By \cref{lemma:mono-epi-left-continuous}, it follows that the constituents of the decomposition remain left continuous with the mono-epi property.
 \end{proof}
\end{corollary}

\subsection{Multiplicative Decomposition for Upper Semicontinuous Modules}

In the upper semicontinuous case, we dually study the decomposition as a direct product, which we refer to as a multiplicative decomposition.
The previously cited \cite{Schmahl22} treats this in the one-parameter case.
To generalize to the $n$-parameter case, we start out with a lemma, whose proof for the one-parameter case in \cite{Schmahl22} once again carries over almost verbatim to the multiparameter case:

\begin{lemma}
  Let $V_i,i\in I$ be a set of persistence modules so that $\prod_{i\in I}V_i$ is q-tame.
  Then the canonical map \[\varphi:\bigoplus_{i\in I}V_i \to \prod_{i\in I}V_i\] is a weak equivalence.

  \begin{proof}
    Clearly, the kernel of $\varphi$ is trivial.
    It remains to show that the cokernel is ephemeral.
    Let $\varepsilon>0$.
    We have the commutative diagram 

\[\begin{tikzcd}
	{\bigoplus_{i\in I} V_{i,s}} & {\prod_{i\in I}V_{i,s}} & {\operatorname{cok} \varphi_s} \\
	{\bigoplus_{i\in I}} & {\prod_{i\in I}V_{i,s-\varepsilon}} & {\operatorname{cok} \varphi_{s-\varepsilon}}
	\arrow["{\varphi_s}", from=1-1, to=1-2]
	\arrow["{p_s}", from=1-2, to=1-3]
	\arrow["{\sigma_{s-\varepsilon,s}}", from=2-1, to=1-1]
	\arrow["{\varphi_{s-\varepsilon}}"', from=2-1, to=2-2]
	\arrow["{\pi_{s-\varepsilon,s}}", from=2-2, to=1-2]
	\arrow["{p_{s-\varepsilon}}"', from=2-2, to=2-3]
	\arrow["{\gamma_{s-\varepsilon,s}}"', from=2-3, to=1-3]
\end{tikzcd}\] where the vertical arrows are given by the structure maps.
Now $p_{s-\varepsilon}$ is epic, so it suffices to show that \[p_s\circ \pi_{s,s-\varepsilon} = \gamma_{s- \varepsilon,s}\circ p_{s-\varepsilon} = 0.\]
But $p_s\circ\varphi_s$ vanishes, so it suffices to show that $\pi_{s-\varepsilon,s}$ factors over $\varphi_s$.
We can insert the respective images in the above diagram to obtain the diagram

\[\begin{tikzcd}
	{\bigoplus_{i\in I} V_{i,s}} & {\prod_{i\in I}V_{i,s}} \\
	{\bigoplus_{i\in I}\operatorname{im} V_{i,s-\varepsilon,s}} & {\prod_{i\in I}\operatorname{im} V_{i,s-\varepsilon,s}} \\
	{\bigoplus_{i\in I}} & {\prod_{i\in I}V_{i,s-\varepsilon}}
	\arrow["{\varphi_s}", from=1-1, to=1-2]
	\arrow[from=2-1, to=1-1]
	\arrow["\psi", from=2-1, to=2-2]
	\arrow[from=2-2, to=1-2]
	\arrow["{\sigma_{s-\varepsilon,s}}", from=3-1, to=2-1]
	\arrow["{\varphi_{s-\varepsilon}}"', from=3-1, to=3-2]
	\arrow["{\pi_{s-\varepsilon,s}}", from=3-2, to=2-2]
\end{tikzcd}\] where $\psi$ is the inclusion of the direct sum into the product.
We assume that $\prod_{i\in I}V_i$ is q-tame, and hence the image of $\pi_{s-\varepsilon,s}$ is finite.
This implies that $\psi$ is an isomorphism, and we can therefore factor the structure map as
\[\begin{tikzcd}
	{\prod_{i\in I}V_{i,s-\varepsilon}} & {\operatorname{im} \pi_{s-\varepsilon,s}} & {\operatorname{im} \sigma_{s-\varepsilon,s}} & {\bigoplus_{i\in I}V_{i,s}} & {\prod_{i\in I}V_{i,s}}
	\arrow[from=1-1, to=1-2]
	\arrow["{\psi^{-1}}", from=1-2, to=1-3]
	\arrow[from=1-3, to=1-4]
	\arrow["{\varphi_s}", from=1-4, to=1-5]
\end{tikzcd}\] as desired.
  \end{proof}
\end{lemma}

\begin{lemma}\label{lemma:weak-equivalence-qtUsc}
  Let $\varphi:V\to W$ be a weak equivalence.
  Then $\underline\varphi$ is an isomorphism.
  
  \begin{proof}
    This follows from the equivalence $\mathbf{Ob}_n\simeq\mathbf{Usc}_n$.
  \end{proof}
\end{lemma}

We now obtain the generalized version of the main theorem of \cite{Schmahl22}, again with an analogous proof, which we include for completeness.

\begin{proposition}\label{proposition:decomposition-qtUsc}

  Let $V\in\mathbf{qtUsc}_n$.
  Then there is a decomposition \[V \cong \prod_{i\in I}V_i\] where the $V_i\in\mathbf{qtUsc}_n$ are indecomposable and unique up to reordering.
  
  \begin{proof}
  We compute \[V\cong\underline{V}\cong \underline{\overline V}\cong \underline{\bigoplus_{i\in I} V_i}\] where we used the equivalence $\mathbf{qtLsc}_n\simeq\mathbf{qtUsc}_n$ and the existence of additive decomposition from this chapter.
  We now observe that $\prod_{i\in I}V_i$ is q-tame.
  This follows from the fact that if \[\prod_{i\in I}\operatorname{im} V_{i,s-\varepsilon,s}\] has infinite dimension then the direct sum has, too.
  Thus \[\varphi:\bigoplus_{i\in I}V_i\to\prod_{i\in I}V_i\] is a weak equivalence, and hence \[\underline{\bigoplus_{i\in I}V_i} \cong \underline{\prod_{i\in I}V_i}\] by~\cref{lemma:weak-equivalence-qtUsc}.
  Limits commute with limits, so
  \[\underline{\prod_{i\in I}V_i} \cong \prod_{i\in I} \underline{V_i},\]
  and together we have
  \[V\cong \prod_{i\in I}\underline{V_i}.\]
  The indecomposability of $\underline{V_i}$ follows from the fact that biproducts commute with limits.
  Uniqueness follows from the equivalence $\mathbf{qtLsc}_n\simeq\mathbf{qtUsc}_n$ and uniqueness of additive decompositions. 
  \end{proof}
\end{proposition}

From the equivalence of categories shown above, we now have the first part of our main theorem: 

\begin{corollary}
  In the observable category of q-tame modules, every object decomposes essentially uniquely into a biproduct of indecomposables with local endomorphism rings.

  \begin{proof}
    In the proof of \cref{proposition:decomposition-qtUsc} we have shown that for $V\in\mathbf{qtUsc}_n$ there is a decomposition \[V\cong \prod_{i\in I}\underline{V_i}\] where the $\underline{V_i}$ are indecomposable and unique up to reordering, which moreover comes from the additive decomposition \[V \cong \bigoplus_{i\in I} V_i.\] But in the observable category we have $V_i \cong \underline{V_i}$, which shows that $V$ decomposes as a biproduct of indecomposables.
  \end{proof}
\end{corollary}

\begin{example}\label{example:no-decomposition}
  We now show that such decompositions do not always exist when we drop the assumption of q-tameness.
  The following example~\cite[Remark 2.9]{strandstab}, originally due to Lesnick, does not admit an interval decomposition.
  Indeed, the argument given there generalizes to show that it does not admit any decomposition into indecomposables.
  We let \[L_0 = \mathbb{k}\] and \[L_{-s} = \bigoplus_{k=s}^\infty \mathbb{k}\] when $s\geq 1$, which we interpret as almost everywhere vanishing sequences with zeros before the $s$-th component.
  For $s\leq t < 0$ we take as structure maps the inclusions, and for the maps $L_s\to L_0$ we take the summation.
  Now consider any decomposition \[L=\bigoplus_{i\in I}M_i\] into indecomposables.
  Among them is exactly one module that is nonvanishing at $s = 0$, say $M$.
  We now show that $M$ is an interval module.
  Assume first that $M$ contains at some first index $t<0$ some nonzero sequence summing to zero. Then this sequence is sent to $0$ at index $s = 0$.
  We can thus write $M = [t,0)\oplus M'$, and by indecomposability of $M$ we must have $M' = 0$.
  If $M$ contains no such sequence, then it can at most be one-dimensional at any index, since the sequences with sum zero have codimension one.
  By indecomposability, it follows that $M$ must be an interval.
  However, since the intersection of all $L_{s}$ with $s\leq -1$ is $0$, it follows that $M$ cannot be $(-\infty,0]$.
  Moreover, we can also conclude that $M\neq [s,0]$, since otherwise the structure map $L_{s-1,0}$ would vanish, leading to a contradiction.
\end{example}

\section{Metric Constructions}

In this section, we discuss the metric aspects of \cref{theorem:main-theorem}; specifically, we show that in the observable category of of q-tame persistence modules, the interleaving distance defines a complete metric on the space and reflects isomorphisms.

\subsection{Completeness}

The goal of this section is to show that the space $\mathbf{qtLsc}_n$ of q-tame lower semicontinuous is complete, thereby establishing (P3) from the main theorem. 

\begin{lemma}
  Let $V^{(k)}$ be a sequence of $n$-parameter modules that is Cauchy for the interleaving distance.
  Then $V^{(k)}$ admits a metric limit under $d_I$, which is constructed as a colimit from the $V^{(k)}$ and the interleaving morphisms.

  \begin{proof}
    Since $V^{(k)}$ is a Cauchy sequence, any accumulation point is already a limit, and we may replace $V^{(k)}$ by a subsequence with summable distances: For every $k$ there is some minimal $N_k$ so that all elements above index $N_k$ have pairwise distance less than $\frac{1}{2^k}$, and the subsequence indexed by the $N_k$ now has summable distances.
    We can thus assume without loss of generality that $V^{(k)}$ and $V^{(k+1)}$ are $\varepsilon_{k}$-interleaved with \[\sum^\infty_{k = 1}\varepsilon_k<\infty.\]
    Write \[\delta_k = \sum_{m\geq k} \varepsilon_m\] for the tail sums.
    We now fix $\varepsilon_{k}$-interleavings $(f_k,g_k)$, where $f_k:V^{(k)}\to V^{(k+1)}[\varepsilon_{k}]$ and $g_k:V^{(k+1)}\to V^{(k)}[\varepsilon_{k}]$.
    We can then form the colimit \[V = \underset{k\geq 1}{\operatorname{colim}}~V^{(k)}[-\delta_k]\] along the morphisms $f_k[-\delta_k]$.
    Since $-\delta_k+\varepsilon_k = -\delta_{k+1}$ the shifted map $f_{k}[-\delta_k]$ defines a map $V^{(k)}[-\delta_k]\to V^{(k+1)}[-\delta_{k+1}]$.
    We now fix $k\geq 1$. We show that $V$ and $V^{(k)}$ are $\delta_k$-interleaved.
    Since $(\delta_k)_k\to 0$ as $k\to\infty$ this then yields the lemma. We first observe that by cofinality of $\{m \geq k\} \subset \mathbb N$, $V$ is the colimit of the diagram 
\[\begin{tikzcd}[row sep=large, column sep=large]
	{V^{(k)}[-\delta_k]} & {V^{(k+1)}[-\delta_{k+1}]} & {V^{(k+2)}[-\delta_{k+2}]} & \dots
	\arrow["{f_k[-\delta_k]}", from=1-1, to=1-2]
	\arrow["{f_{k+1}[-\delta_{k+1}]}", from=1-2, to=1-3]
	\arrow[from=1-3, to=1-4]
\end{tikzcd}\]
The diagram 
\[\begin{tikzcd}[row sep=large, column sep=large]
	{V^{(k)}[-\delta_k]} & {V^{(k+1)}[-\delta_{k+1}]} & {V^{(k+2)}[-\delta_{k+2}]} & \dots \\
	{V^{(k)}[\delta_k]} & {V^{(k+1)}[\delta_{k+1}]} & {V^{(k+2)}[\delta_{k+2}]} & \dots
	\arrow["{f_k[-\delta_k]}", from=1-1, to=1-2]
	\arrow["{\eta_{2\delta_k}}"', from=1-1, to=2-1]
	\arrow["{f_{k+1}[-\delta_{k+1}]}", from=1-2, to=1-3]
	\arrow["{\eta_{2\delta_{k+1}}}"', from=1-2, to=2-2]
	\arrow[from=1-3, to=1-4]
	\arrow["{\eta_{2\delta_{k+2}}}"', from=1-3, to=2-3]
	\arrow["{g_k[\delta_{k+1}]}", from=2-2, to=2-1]
	\arrow["{g_{k+1}[\delta_{k+2}]}", from=2-3, to=2-2]
	\arrow[from=2-4, to=2-3]
\end{tikzcd}\] commutes by the interleaving relations $g_m[\varepsilon_m]\circ f_m = \eta_{2\varepsilon_m}$ for $m\geq k$.
Composing the vertical maps with the horizontal maps in the bottom row yields a cocone from the top row to $V^{(k)}[\delta_k]$, and the universal property of $V$ as the colimit of the top row gives us a map $f:V\to V^{(k)}[\delta_k]$.
%
%
%
The diagram 
\[\begin{tikzcd}[row sep=large, column sep=huge]
	{V^{(k)}} & {V^{(k+1)}[\delta_{k+1}-\delta_k]} & {V^{(k+2)}[\delta_{k+2}-\delta_k]} & \dots \\
	{V^{(k)}} & {V^{(k+1)}[\delta_k-\delta_{k+1}]} & {V^{(k+2)}[\delta_k-\delta_{k+2}]} & \dots
	\arrow["{\operatorname{id}}", from=1-1, to=2-1]
	\arrow["{g_k[\delta_{k+1}-\delta_k]}"', from=1-2, to=1-1]
	\arrow["{\eta_{2(\delta_k - \delta_{k+1})}}", from=1-2, to=2-2]
	\arrow["{g_{k+1}[\delta_{k+2}-\delta_k]}"', from=1-3, to=1-2]
	\arrow["{\eta_{2(\delta_k - \delta_{k+2})}}", from=1-3, to=2-3]
	\arrow[from=1-4, to=1-3]
	\arrow["{f_{k}}"', from=2-1, to=2-2]
	\arrow["{f_{k+1}[\delta_k - \delta_{k+1}]}"', from=2-2, to=2-3]
	\arrow[from=2-3, to=2-4]
\end{tikzcd}\] commutes; this follows from the interleaving relations $f_m[\varepsilon_m]\circ g_m = \eta_{2\varepsilon_m}$ for $m\geq k$.
Observe that the bottom row has colimit $V[\delta_k]$.
Once more, this diagram yields a morphism $g:V^{(k)}\to V[\delta_k]$.
Shifting the second diagram by $\delta_k$, we can stack both diagrams vertically: 
\[\begin{tikzcd}[row sep=large, column sep=huge]
	{V^{(k)}[-\delta_k]} & {V^{(k+1)}[-\delta_{k+1}]} & {V^{(k+2)}[-\delta_{k+2}]} & \dots \\
	{V^{(k)}[\delta_k]} & {V^{(k+1)}[\delta_{k+1}]} & {V^{(k+2)}[\delta_{k+2}]} & \dots \\
	{V^{(k)}[2\delta_k]} & {V^{(k+1)}[2\delta_k-\delta_{k+1}]} & {V^{(k+2)}[2\delta_k-\delta_{k+2}]} & \dots
	\arrow["{f_k[-\delta_k]}", from=1-1, to=1-2]
	\arrow["{\eta_{2\delta_k}}"', from=1-1, to=2-1]
	\arrow["{f_{k+1}[-\delta_{k+1}]}", from=1-2, to=1-3]
	\arrow["{\eta_{2\delta_{k+1}}}"', from=1-2, to=2-2]
	\arrow[from=1-3, to=1-4]
	\arrow["{\eta_{2\delta_{k+2}}}"', from=1-3, to=2-3]
	\arrow["{\operatorname{id}}", from=2-1, to=3-1]
	\arrow["{g_k[\delta_{k+1}]}", from=2-2, to=2-1]
	\arrow["{\eta_{2(\delta_k - \delta_{k+1})}}", from=2-2, to=3-2]
	\arrow["{g_{k+1}[\delta_{k+2}]}", from=2-3, to=2-2]
	\arrow["{\eta_{2(\delta_k - \delta_{k+2})}}", from=2-3, to=3-3]
	\arrow[from=2-4, to=2-3]
	\arrow["{f_k[2\delta_k]}"', from=3-1, to=3-2]
	\arrow["{f_{k+1}[2\delta_k-\delta_{k+1}]}"', from=3-2, to=3-3]
	\arrow[from=3-3, to=3-4]
\end{tikzcd}\]
Each vertical column composes to the corresponding shift map $\eta_{2\delta_k}$.
The maps $f$ and $g$ were obtained from universal constructions, hence $g[\delta_k]\circ f$ is obtained in the same way from this stacked diagram.
The colimit of shifts is again a shift, so altogether we obtain $g[\delta_k]\circ f= \eta_{2\delta_k}$.
Stacking the diagrams in reversed order by shifting the first one gives us $f[\delta_k]\circ g = \eta_{2\delta_k}$ and hence the desired $\delta_k$-interleaving between $V$ and $V^{(k)}$. 
\end{proof}
\end{lemma}

\begin{remark}
  The construction can also be carried out via a limit instead of a colimit, reversing the roles of the $f_k$ and $g_k$. 
\end{remark}

It now remains to verify that the limit we have constructed remains in $\ourspace{n}$ to obtain the result.
We will use the following lemma:

\begin{lemma}
\label{lemma:q-tame-shift}
A module $V$ is q-tame if for any $\varepsilon > 0$ 
there is some q-tame module $W$ so that
the shift $\eta_\varepsilon$ factorizes over $W$.
\begin{proof}
Let $\varepsilon>0$. We show that $V_{s-\varepsilon,s}$ has finite rank.
By assumption we have a factorization 
\[\begin{tikzcd}
  V && {V[\frac{\varepsilon}{2}]} \\
  & W
  \arrow["{\eta_{\frac{\varepsilon}{2}}}", from=1-1, to=1-3]
  \arrow[from=1-1, to=2-2]
  \arrow[from=2-2, to=1-3]
\end{tikzcd}
\]
where $W$ is q-tame.
We obtain the commutative diagram
\[\begin{tikzcd}
	{V_{s-\varepsilon}} & {W_{s-\varepsilon}} & {V_{s-\frac{\varepsilon}{2}}} \\
	{V_{s-\frac{\varepsilon}{2}}} & {W_{s-\frac{\varepsilon}{2}}} & {V_{s}}
	\arrow[from=1-1, to=1-2]
	\arrow[from=1-1, to=2-1]
	\arrow[from=1-2, to=1-3]
	\arrow[from=1-2, to=2-2]
	\arrow[from=1-3, to=2-3]
	\arrow[from=2-1, to=2-2]
	\arrow[from=2-2, to=2-3]
\end{tikzcd}\] Now since $W$ is q-tame the structure map $W_{s-\varepsilon}\to W_{s-\frac{\varepsilon}{2}}$ has finite rank and $V_{s-\varepsilon}\to V_s$ factors over it.
Since $\varepsilon$ was arbitrary this yield the lemma.
\end{proof}
\end{lemma}

\begin{proposition}
  The space $\ourspace{n}$ is complete under the interleaving distance.

  \begin{proof}
    Since colimits commute with colimits we immediately see that the limit we have constructed above remains lower semicontinuous.
    Furthermore it follows from Lemma~\ref{lemma:q-tame-shift} that $V$ is q-tame, since the $V^{(k)}$ are q-tame and the $\delta_k$-interleavings provide factorizations for arbitrarily small shifts. 
  \end{proof}
\end{proposition}

\subsection{Upper Semicontinuity of the Persistent Set of Interleavings}

Our goal for the next two sections is the existence of an interleaving which realizes the interleaving distance: if $d_I(V,W) = \delta$ in the space of q-tame lower semicontinuous modules then we want to produce a $\delta$-interleaving of $V$ and $W$.
As a corollary we will then immediately obtain the second part of \cref{theorem:main-theorem}.
The current section shows that interleavings of lower semicontinuous persistence modules form an upper semicontinuous persistent set, following a strategy sketched in \cite[Section 4.5]{scoccola}.

\begin{notation}
  Let $V,W$ be fixed persistence modules. Let $r\geq 0$. We write \[\mathbb{I}(V,W)(r) = \{r\textrm{-interleavings of }V\textrm{ and }W\}.\]
  For $r<r'$ we have a map $\mathbb{I}(V,W)(r)\to\mathbb{I}(V,W)(r')$ via $(f,g)\mapsto (\eta_{r'-r}\circ f,\eta_{r'-r}\circ g)$.
  This gives $\mathbb{I}(V,W)$ the structure of a persistent set over $\mathbf{R}_{\geq 0}$.
\end{notation}

  Let $V,W$ be fixed persistence modules.
  We consider the assignment \[\varepsilon\mapsto \operatorname{Hom}(V,W[\varepsilon]).\]
  Whenever $\varepsilon <\varepsilon'$, we can transform any morphism $V\to W[\varepsilon]$ to a morphism $V\to W[\varepsilon']$ by postcomposition with $\eta_{\varepsilon'-\varepsilon}$.
  Thus $\operatorname{Hom}(V,W[\cdot])$ is a persistent set over $\mathbf{R}_{\geq 0}$.

\begin{proposition}\label{prop:upper-semicontinuity-interleavings}
  If $V$ is lower semicontinuous or $W$ is upper semicontinuous, then
  the persistent set $\mathbb{I}(V,W)$ is upper semicontinuous.
  \begin{proof}
    Consider the morphisms \[\xi_\varepsilon:\operatorname{Hom}(V,W[\varepsilon])\times\operatorname{Hom}(W,V[\varepsilon])\to \operatorname{Hom}(V,V[2\varepsilon])\times \operatorname{Hom}(W,W[2\varepsilon])\] \[(f,g)\mapsto (g[\varepsilon]\circ f,f[\varepsilon]\circ g)\] which defines a morphism $\xi$ of the respective persistence sets.
    We denote by $\operatorname{pt}$ the terminal object in the category of persistent sets; it has a unique element $*$ at every index.
    We define a morphism $\operatorname{pt}\to \operatorname{Hom}(V,V[2\varepsilon])\times \operatorname{Hom}(W,W[2\varepsilon])$ at $\varepsilon\geq 0$ by mapping ${*}\mapsto (\eta_{2\varepsilon}^V,\eta^W_{2\varepsilon})$.
    Since limits of functors are computed levelwise we get a pullback square 

\[\begin{tikzcd}
	{\mathbb{I}(V,W)} & {\operatorname{Hom}(V,W[\cdot])\times\operatorname{Hom}(W,V[\cdot])} \\
	{\operatorname{pt}} & {\operatorname{Hom}(V,V[2\cdot])\times \operatorname{Hom}(W,W[2\cdot])}
	\arrow[from=1-1, to=1-2]
	\arrow[from=1-1, to=2-1]
	\arrow[from=1-2, to=2-2]
	\arrow[from=2-1, to=2-2]
\end{tikzcd}\] Since limits commute with limits it now suffices to show that the persistent sets involved in the pullback diagram are upper semicontinuous.
For this we only need to show that $\operatorname{Hom}(V,W[\cdot])$ (and dually, $\operatorname{Hom}(W,V[\cdot])$) is upper semicontinuous at any $\delta\geq 0$.
In the case that $V$ is lower semicontinuous, we have \[ V[\delta] \cong \underset{\varepsilon>\delta}{\operatorname{colim}}~V[-\varepsilon],\] 
and since the $\operatorname{Hom}$ functor sends limits to colimits in the first variable, we get
\[\lim_{\varepsilon>\delta}\operatorname{Hom}(V,W[\varepsilon]) = \lim_{\varepsilon>\delta}\operatorname{Hom}(V[-\varepsilon],W) \cong \operatorname{Hom}(\underset{\varepsilon>\delta}{\operatorname{colim}}~V[-\varepsilon],W).  \] In the case that $W$ is upper semicontinuous, we dually use the fact that the Hom-functor preserves limits in the second variable to obtain upper semicontinuity of the Hom-persistence set by an analogous argument.
  \end{proof}
\end{proposition}

We therefore have that
\[\mathbb{I}(V,W)(\delta)  \cong \lim_{\varepsilon>\delta}~\mathbb{I}(V,W)(\varepsilon)\]
and thus by a cofinality argument
\[\mathbb{I}(V,W)(\delta) \cong \lim_{n\geq 1} \mathbb{I}(V,W)\left(\delta + \textstyle \frac{1}{n}\right)\]
(or any other countable cofinal set).
In general, the limit of infinite nonempty sets can be empty, so this is not yet sufficient for us to conclude $\mathbb{I}(V,W)(\delta) \neq \emptyset$ whenever $d_I(V,W) = \delta$.
In the next section we will turn $\mathbb{I}(V,W)$ into a persistent topological space to address this issue.

\subsection{Existence of Metric-Realizing Interleavings}

In this section we show that the interleaving distance indeed reflects isomorphisms in the q-tame and observable setting.
More generally, we establish that if $d_I(V,W) = \delta$, then there is a $\delta$-interleaving between $V$ and $W$. 
Recall that \cref{prop:upper-semicontinuity-interleavings}  establishes that any $\delta$-interleaving is a limit of $\varepsilon$-interleavings for $\varepsilon>\delta$.
Conversely, however, it is not clear a priori that the limit of $\varepsilon$-interleavings exists, since the limit of a tower \[\dots\to X_3\to X_2\to X_1\] of infinite nonempty sets can be empty, as illustrated by the following simple example.

\begin{example}
  Let $X_n = \{n,n+1,n+2,\dots\}$ with the inclusions \[\dots\to X_3\to X_2\to X_1.\]
  We then have \[\lim_{k\geq 1} X_k = \bigcap^\infty_{k=1}X_k = \emptyset , \] even though every individual $X_n$ is nonempty. 
\end{example}

We thus need some additional properties to guarantee that the limit is inhabited.
Our strategy for this section is to apply the following theorem to obtain a suitable interleaving: 

\begin{theorem}[Stone~\cite{stone}]
  Consider the poset $\mathbf{N}^{\operatorname{op}} = (\mathbb N, \geq)$.
  Let $X:\mathbf{N}^{\operatorname{op}}\to\mathbf{Top}$ be a tower of topological spaces such that \begin{enumerate}
    \item every $X_s$ is compact and $T_0$, and
    \item every $X_{st}$ is closed.
  \end{enumerate}
  Then \(\lim X_\bullet \neq \emptyset.\)
\end{theorem}

We first show the following corollary, which allows us to replace the condition of compactness of spaces by compactness of the images of the structure maps.

\begin{corollary}\label{cor:stone_q_tame}
  If $X:\mathbf{N}^{\operatorname{op}}\to\mathbf{Top}$ is a diagram of topological spaces such that
  \begin{enumerate}
    \item $X$ is \emph{q-tame}, that is, all nontrivial structure maps have compact and $T_0$ image, and
    \item all structure maps are closed. 
  \end{enumerate}
  Then \[\lim X_\bullet\neq\emptyset.\]

  \begin{proof}
    We let $\widehat X_s = \operatorname{im} X_{s+1,s}\subseteq X_s$ and define $\widehat X_{s+1,s}$ by restricting $X_{s+1,s}:X_{s+1}\to X_s$ to $\operatorname{im}(X_{s+2,s+1})$. By construction, this map lands in $\widehat X_s$.
    The diagram $\widehat X$ satisfies the conditions of the theorem of Stone by construction. Hence \[\emptyset\neq \lim \widehat X_\bullet = \lim X_\bullet\] which shows the corollary.
  \end{proof}
\end{corollary}

The goal is to apply \cref{cor:stone_q_tame}to the tower of interleavings $\mathbb{I}(V,W)(\delta + \frac{1}{n})$, topologized as a subspace of the product of the Hom-sets $\operatorname{Hom}(V,W[\delta + \frac{1}{n}])\times\operatorname{Hom}(W,V[\delta + \frac{1}{n}])$.
In \cite[Theorem 4.5.14]{scoccola}, it is shown that if the Hom-sets admint suitable topologies, then the existence of metric-realizing interleavings follows from the upper semicontinuity of the persistent set of interleavings and the theorem of Stone. 
Our approach is similar; however, as explained later in \cref{rmk:topologize-product}, we have to deviate from this approach slightlyand equip the product of the Hom-sets with a topology different from the product topology.

\begin{proposition}
\label{prop:hom-space-topologies}
  Assume that all sets \[\operatorname{Hom}(V,W)\times\operatorname{Hom}(V',W')\] are equipped with topologies so that
  \begin{enumerate}
    \item all spaces are $T_1$,
    \item the maps $\xi_\varepsilon$ are continuous, and
    \item postcomposition with a pair of shifts \[S_\varepsilon = (\eta_\varepsilon\circ -,\eta_\varepsilon\circ -):\operatorname{Hom}(V,W)\times\operatorname{Hom}(V',W')\to \operatorname{Hom}(V,W[\varepsilon])\times\operatorname{Hom}(V',W'[\varepsilon])\] is continuous, closed, and has compact image.
  \end{enumerate} 
If $V,W\in \mathbf{qtLsc}_n$ satisfy $d_I(V,W) = \delta$, then $\mathbb{I}(V,W)(\delta) \neq \emptyset$. 

\begin{proof}
  The goal is to be able to apply Stone's theorem to $\mathbb{I}(V,W)$.
  For $\varepsilon>0$ we again consider 
  \[\begin{tikzcd}
    {\mathbb{I}(V,W)(\varepsilon)} & {\operatorname{Hom}(V,W[\varepsilon])\times \operatorname{Hom}(W,V[\varepsilon])} \\
    {\operatorname{pt}} & {\operatorname{Hom}(V,V[2\varepsilon])\times \operatorname{Hom}(W,W[2\varepsilon])}
    \arrow[from=1-1, to=1-2]
    \arrow[from=1-1, to=2-1]
    \arrow[from=1-2, to=2-2]
    \arrow[from=2-1, to=2-2]
  \end{tikzcd}\] which is a pullback square.
  The map $\xi_\varepsilon$ is continuous, so we can topologize $\mathbb{I}(V,W)(\varepsilon)$ with the limit topology.
  By the $T_1$-property $\mathbb{I}(V,W)(\varepsilon)$ is closed in $\operatorname{Hom}(V,W[\varepsilon])\times \operatorname{Hom}(W,V[\varepsilon])$ as the preimage of the pair of shifts by $2\varepsilon$.
  Postcomposition with pair of shifts $(\eta_{\varepsilon'}^W,\eta_{\varepsilon'}^V)$ is continuous, closed, and has compact image.
  Since $\mathbb{I}(V,W)(\varepsilon)$ is closed, the same holds there.
  We conclude that it satisfies the conditions of \cref{cor:stone_q_tame}, and we can conclude $\mathbb{I}(V,W)(\delta)\neq \emptyset$. 
\end{proof}
\end{proposition}
We now show that we can find topologies on the Hom-sets which satisfy the conditions of \cref{prop:hom-space-topologies}. We first reduce to the case of modules with bounded support, from which we can then obtain the general case by a limit construction.

\begin{notation}\label{def:hom-lim}
  Let $k\geq 1$ and let $V,W$ be $n$-parameter persistence modules.
  We let \[\operatorname{Hom}(V,W)\vert_k = \operatorname{Hom}(V\vert_{\mathbf{C}^k},W\vert_{\mathbf{C}^k})\] where $\mathbf{C}^k = [-k,k]^n\subseteq\mathbf{R}^n$.
  As sets we have \[\operatorname{Hom}(V,W) \cong \lim_{k\geq 1}\operatorname{Hom}(V,W)\vert_k\] with the limit being taken along the restriction maps. 
\end{notation}

\begin{lemma}\label{lem:hom-space-topologies-bounded}
  Assume that for each $k \in \mathbb N$ the Hom-sets $\operatorname{Hom}(V,W)\vert_k\times\operatorname{Hom}(V',W')\vert_k$ are equipped with topologies so that the conditions of \cref{prop:hom-space-topologies} hold and so that the restrictions \[\operatorname{Hom}(V,W)\vert_{k+1}\times\operatorname{Hom}(V',W')\vert_{k+1}\to\operatorname{Hom}(V,W)\vert_k\times\operatorname{Hom}(V',W')\vert_k\] are continuous and closed.
  Then there exist topologies on $\operatorname{Hom}(V,W)\times\operatorname{Hom}(V',W')$ which satisfy the conditions of \cref{prop:hom-space-topologies}.
  \begin{proof}
    As stated in \cref{def:hom-lim}, we have \[\operatorname{Hom}(V,W)\times \operatorname{Hom}(V',W') = \lim_{k\geq 1}(\operatorname{Hom}(V,W)\vert_k\times \operatorname{Hom}(V',W')\vert_k)\] as sets.
    We can hence topologize with the usual limit topology.
    The product of $T_1$-spaces remains $T_1$.
    This implies that the limit, as a subspace of a $T_1$-space, is $T_1$ as well.

    For continuity of $\xi_\varepsilon$, observe that 
    \[\begin{tikzcd}
      {\operatorname{Hom}(V,W[\varepsilon])\times\operatorname{Hom}(W,V[\varepsilon])} & {\operatorname{Hom}(V,W[\varepsilon])\vert_k\times\operatorname{Hom}(W,V[\varepsilon])\vert_k} \\
      {\operatorname{Hom}(V,V[2\varepsilon])\times\operatorname{Hom}(W,W[2\varepsilon])} & {\operatorname{Hom}(V,V[2\varepsilon])\vert_k\times\operatorname{Hom}(W,W[2\varepsilon])\vert_k}
      \arrow[from=1-1, to=1-2]
      \arrow["{\xi_\varepsilon}"', from=1-1, to=2-1]
      \arrow["{\xi_\varepsilon^k}", from=1-2, to=2-2]
      \arrow[from=2-1, to=2-2]
    \end{tikzcd}\] commutes. 
    Hence $\xi_\varepsilon$ is continuous as a map constructed from continuous limit data.

    For postcomposition with shifts (i.e., the maps $S_\varepsilon$ from \cref{prop:hom-space-topologies}), we observe that
    \[\begin{tikzcd}
      {\operatorname{Hom}(V,W)\times\operatorname{Hom}(V',W')} & {\operatorname{Hom}(V,W)\vert_k\times\operatorname{Hom}(V',W')\vert_k} \\
      {\operatorname{Hom}(V,W[\varepsilon])\times\operatorname{Hom}(V',W'[\varepsilon])} & {\operatorname{Hom}(V,W[\varepsilon])\vert_k\times\operatorname{Hom}(V',W'[\varepsilon])\vert_k}
      \arrow[from=1-1, to=1-2]
      \arrow["{S_\varepsilon}", from=1-1, to=2-1]
      \arrow["{S_\varepsilon^k}"', from=1-2, to=2-2]
      \arrow["{\operatorname{res}_k}", from=2-1, to=2-2]
    \end{tikzcd}\] commutes, so the same reasoning implies that the maps obtained in this manner are continuous as well.
    This diagram also implies that \[\operatorname{im} S_\varepsilon \cong  \lim_{k\geq 1} \operatorname{im} S_\varepsilon^k\] and more concretely \[\operatorname{im} S_\varepsilon \cong \bigcap_{k=1}^\infty\operatorname{res}_k^{-1}(\operatorname{im} S_\varepsilon^k)\] which means that the image of $S_\varepsilon$ is closed as an intersection of closed sets.
    It follows that \[\prod^\infty_{k=1}\operatorname{res}_k:\operatorname{im} S_\varepsilon \to \prod^\infty_{k=1}\operatorname{im} S_\varepsilon^k\] is a closed embedding and that the right hand side is compact as a product of compact spaces.
    Hence compactness of the image follows.
  \end{proof}
\end{lemma}

It remains to construct such topologies for the restricted Hom-sets $\operatorname{Hom}(V,W)\vert_k \times \operatorname{Hom}(V',W')\vert_k$.
For this, we recall the Zariski topology~\cite[Chapter 1]{hartshorne77}.

\begin{definition}
  Let $V \cong \mathbb{k}^d$ be a finite-dimensional $\mathbb{k}$-vector space. We equip $\mathbb{k}$ with the cofinite topology, which in turn equips $\mathbb{k}^d$ with the coarsest topology which makes all polynomials in $\mathbb{k}[x_1,\dots,x_d]$ continuous.
  We call it the \emph{Zariski topology}. 
\end{definition}

We now extend the above definition to infinite-dimensional spaces (see for example \cite{Kumar2002}). 
\begin{definition}
Let $V$ be a (not necessarily finite-dimensional) $\mathbb{k}$-vector space.
As sets we can write
\[V \cong\underset{W\subseteq V\textrm{ finite-dimensional}}{\operatorname{colim}}~W,\]
which allows us to topologize $V$.
We call this the \emph{ind-Zariski topology}.
\end{definition}
Let $V$ and $W$ be vector spaces equipped with the ind-Zariski topology and let $\varphi:V\to W$ be linear.
If $U\subseteq V$ is finite-dimensional, then $\varphi\vert_U :U\to \varphi(U)$ is continuous.
By functoriality of the colimit this implies that $\varphi$ is continuous.

\begin{definition}
  We say that a map $\varphi:\mathbb{k}^d\to\mathbb{k}^{d'}$ is \emph{polynomial} if it is defined by a polynomial in every component $j\in \{1,\dots,d'\}$.
  For finite-dimensional vector spaces we say that a map is polynomial if it is polynomial for a suitable choice of basis.
  Finally, for arbitrary vector spaces we call a map \emph{polynomial} if it can be written as a colimit of polynomial maps on finite-dimensional subspaces.
  Polynomial maps are continuous in the ind-Zariski topology. 
\end{definition}

\begin{remark}\label{rmk:topologize-product}
  For our purposes, the main problem with the Zariski topology, and the reason that we topologize the product directly, is that bilinear maps are not continous for the product topology if the underlying field is infinite.
  To see this, consider the multiplication \[\mu:\mathbb{k}\times\mathbb{k}\to\mathbb{k}.\]
  Then on the left we have a product of cofinite topologies.
  But the preimage $\mu^{-1}(0) = \{xy = 0\}$ is not finite if $\mathbb{k}$ is infinite, and hence not closed.
  Unfortunately, this suggests that there is no easy way to fulfill the criteria in \cite{scoccola} except in the case of a finite base field. 
\end{remark}

The following lemma is the key to showing that the topologies we have defined on the Hom-sets satisfy the requirement of compact images of the shift maps, since finite-dimensional subspaces are compact and closed in the ind-Zariski topology.

\begin{lemma}\label{lem:fin-dim-image}
  The image of \[\operatorname{Hom}(V,W)\vert_k \to \operatorname{Hom}(V,W[\varepsilon])\vert_k\] is finite-dimensional.

  \begin{proof}
    The image consists of precisely those $\psi:V\to W[\varepsilon]$ which have a factorization 
    \[\begin{tikzcd}
      V & W \\
      & {W[\varepsilon]}
      \arrow["\varphi", from=1-1, to=1-2]
      \arrow["\psi"', from=1-1, to=2-2]
      \arrow["{\eta_\varepsilon^W}", from=1-2, to=2-2]
    \end{tikzcd}\] over an $\varepsilon$-shift.
    Let $s\in \mathbf{C}_k$. For any $s\ll t \ll s+\varepsilon$ we have a commutative diagram 
\[\begin{tikzcd}
	{V_s} & {\operatorname{im} V_{st}} & {V_{s+\varepsilon}} \\
	{W_s} & {\operatorname{im} W_{st}} & {W_{s+\varepsilon}}
	\arrow[from=1-1, to=1-2]
	\arrow["{\varphi_s}"', from=1-1, to=2-1]
	\arrow[from=1-2, to=1-3]
	\arrow["{\alpha_t}"', from=1-2, to=2-2]
	\arrow["{\varphi_{s+\varepsilon}}", from=1-3, to=2-3]
	\arrow[from=2-1, to=2-2]
	\arrow[from=2-2, to=2-3]
\end{tikzcd}\] where $\alpha_t$ is the linear map $\varphi_t$, suitably restricted. Observe that $\alpha_t$ is a map between finite-dimensional vector spaces. From the commutativity, we get a factorization  \[\psi_s = W_{t,s+\varepsilon}\circ\alpha_t\circ V_{st}.\]  We now fix any uniform grid $\mathbf{Q}\subseteq\mathbf{C}_k$ with width $\leq\frac{\varepsilon}{2}$. Since the $k$-cube is compact, the grid is finite. For any $s$ there is some $t\in\mathbf{Q}$ with $s\ll t\ll s+\varepsilon$, allowing us to write $\psi_s$ in terms of~$\alpha_t$ and the structure maps of $V$ and $W$. Thus $\psi$ is completely determined by the family $(\alpha_t)_{t\in\mathbf{Q}}$, and each $\psi_s$ has finite rank, proving the claim.
  \end{proof}
\end{lemma}

We can now proceed with the proof of the main result of this section, which is that the ind-Zariski topology on the Hom-sets satisfies the conditions of \cref{lem:hom-space-topologies-bounded}.

\begin{lemma}
  Topologizing $\operatorname{Hom}(V,W)\vert_k\times \operatorname{Hom}(V',W')\vert_k$ with the ind-Zariski topology  yields topologies which satisfy the conditions of \cref{lem:hom-space-topologies-bounded}. 

  \begin{proof}

The ind-Zariski topology is $T_1$ as points are closed in finite-dimensional subspaces, which are closed in the ambient space.

Observe that the maps $\xi_\varepsilon$ are bilinear, which implies that when choosing coordinates locally it becomes a polynomial expression of degree $2$.
This suffices for continuity under the ind-Zariski topology.

The continuity of pairs of shifts follows from their linearity. 
Now \cref{lem:fin-dim-image} implies that a pair of shifts \[S_\varepsilon : \operatorname{Hom}(V,W)\vert_k\times\operatorname{Hom}(V',W')\vert_k\to \operatorname{Hom}(V,W[\varepsilon])\vert_k\times\operatorname{Hom}(V',W'[\varepsilon])\vert_k\] also has a finite-dimensional image as the product of two finite-dimensional spaces.
It follows that the image must be closed and compact in the ind-Zariski topology. 
  \end{proof}
\end{lemma}

We thus obtain the second part of our main theorem:

\begin{corollary}
  The space of q-tame lower semicontinuous modules is a complete extended metric space.
  In particular, $d_I$ reflects isomorphisms. 
\end{corollary}

Together with the decomposition result from \cref{section:decomposition} we have shown our main theorem:

\begin{theorem}
  Let $n \geq 1$.
  The observable category of q-tame $n$-parameter persistence modules satisfies the following properties:
  \begin{enumerate}
  \item[(P1)] The category is Abelian and Krull--Schmidt, in the sense that every object decomposes as a direct sum of indecomposables in an essentially unique way.
  \item[(P2)] Two persistence modules are isomorphic if and only if they are at interleaving distance zero.
  \item[(P3)] The extended metric induced by the interleaving distance is complete, in the sense every Cauchy sequence of persistence modules has a limit.
  \end{enumerate}
\end{theorem}

\subsection{Results on Separability and Compactness}

We just saw that the observable ismorphism classes of q-tame persistence modules form a complete (extended) metric space. We now study metric and topological properties of this space and of important subspaces. We first exhibit a density result, allowing us to approximate any q-tame module via restriction-extensions to discrete regular grids.

\begin{definition}
  Let $V\in\mathbf{qtUsc}_n$.
  We say that $V$ is \emph{locally finitely presentable} if it is pointwise finite-dimensional and if there is some countable grid $\mathbf{Q}$ so that $V \cong V_\mathbf{Q}$.
  Additionally, if $\mathbf{Q}$ can be chosen to be regular, then we say that $V$ is \emph{regularly locally finitely presentable}.
\end{definition}

\begin{lemma}\label{lem:regularly-locally-fp-dense}
  The regularly locally finitely presentable modules are dense in $\mathbf{qtUsc}_n$.

  \begin{proof}
    Let $V$ be q-tame and  upper semicontinuous.
    We let \[S_\varepsilon V = \operatorname{im}(\eta_\varepsilon:V\to V[\varepsilon])\] for all $\varepsilon>0$. Observe that this module is $\varepsilon$-interleaved with $V$; 
    the interleaving morphisms are given by the submodule inclusion $S_\varepsilon V \hookrightarrow V[\varepsilon]$ and the composition of the quotient map $V\to S_\varepsilon V$ with the shift map $S_\varepsilon V\to S_\varepsilon V[\varepsilon]$.
    Furthermore, $S_\varepsilon V$ is pointwise finite-dimensional since $V$ is q-tame. It thus follows that the pointwise finite-dimensional upper semicontinuous modules are dense in $\mathbf{qtUsc}_n$.

    We may thus assume that $V$ is pointwise finite-dimensional.
    Let $\varepsilon>0$ and consider the regular grid $\mathbf{Q} = (\varepsilon\mathbb{Z})^n$.
    Observe that the module $V_\mathbf{Q}$ is regularly finitely presentable since restriction-extension is idempotent. 
    Composing the natural map $V_\mathbf{Q}\to V$ with the shift map $V\to V[\varepsilon]$, we get a map \[g:V_\mathbf{Q}\to V[\varepsilon].\]
    Dually, we get a map \[f:V\to V_\mathbf{Q}[\varepsilon]\] as follows: 
    For any $s$ there is a unique grid point $p$ such that $s \leq p < s+\varepsilon$.
    By definition of $V_\mathbf{Q}$, this yields a map $V_s\to V_p = (V_\mathbf{Q})_{s+\varepsilon}$ assembled from the structure maps of $V$. These maps define an interleaving of $V$ and $V_\mathbf{Q}$ since they both are defined purely in terms of the structure maps of $V$.
    We can thus conclude that \(d_I(V,V_\mathbf{Q})\leq\varepsilon\), and so $V$ can be approximated by a regularly finitely presentable module in the interleaving distance.
  \end{proof}
\end{lemma}

All persistence modules arising from applications in topological data analysis are finitely presentable. However, the space of finitely presentable modules is not complete, so it is natural to ask for a characterization of its closure in order to study stability properties of methods using these modules.

\begin{proposition}
  The closure of $\mathbf{fp}_n$ under the interleaving distance consists of precisely those modules $V\in\mathbf{qtUsc}_n$ for which $d_I(V,V_{\mathbf{C}_k})\to 0$ as $k\to\infty$.

  \begin{proof}
    Let $V \colon \mathbf{R}^n\to \mathbf{Vect}$ be in the closure of the finitely presentable modules.
    Since $\mathbf{fp}_n$ is in $\mathbf{qtUsc}_n$ which is complete, it follows that $V\in\mathbf{qtUsc}_n$.
    We write $V_k = V_{\mathbf{C}_k}$.
    There is a sequence $W_n$ of finitely presentable modules with $W_k\to V$ as $n\to\infty$.
    We assume that $W_k$ is defined on some finite grid $\mathbf{P}_k$ which we may assume to be contained in the box $[-k,k]^n$.
    By the triangle inequality we get \[d_I(V_k,V)\leq d_I(W_k,V)+d_I(V_k,W_k)\] so it remains to show that $d_I(V_k,W_k)\to 0 $ as $k\to\infty$.
    Let $f:V\to W_k[\varepsilon]$ and $g:W_k\to V[\varepsilon]$ be an $\varepsilon$-interleaving.
    Since $W_k$ is defined on a grid inside $[-k,k]^n$ we can obtain factorizations \[f:V\to V_k\to W_k[\varepsilon]\] and \[g:W_k\to V_k[\varepsilon]\to V[\varepsilon]\] which combine to an $\varepsilon$-interleaving of $W_k$ and $V_k$.
    This shows the claim.
    Conversely, consider a $V\in\mathbf{qtUsc}_n$ with the property \[d_I(V,V_k)\to 0\] as $k\to\infty$.
    This immediately implies that $V_k\to V$, so it remains to show that $V_k\in\overline{\mathbf{fp}_n}$.
    As above, the $V_k$ can be approximated by modules of the form \[ \operatorname{im}(\eta_\varepsilon :V_k\to V_k[\varepsilon])\] which are defined on the box $[-k,k]^n$ and pointwise finite-dimensional.
    Additionally, these can moreover be approximated by a module defined on a regular grid.
    But the intersection of a regular grid with a compact set is finite, hence they can be approximated by pointwise finite-dimensional modules defined on a finite grid.
    By the lemma in the first chapter, these are finitely presentable.
    It follows that $V_k\in\overline{\mathbf{fp}_n}$. 
  \end{proof}
\end{proposition}

\begin{corollary}
  \label{corollary:characterization-closure-finitely-presented}
  The closure $\overline{\mathbf{fp}_n}$ of the finitely presentable modules in the interleaving distance is precisely the space of q-tame, upper semicontinuous modules $V$ such that there is some bounded subset $\mathbf{B}\subseteq\mathbf{R}^n$ with $V_\mathbf{B} \cong V$.
\end{corollary}

\begin{definition}
  We call such modules \emph{boundedly generated} and write $\mathbf{bg}_n$ for the full subcategory of boundedly generated modules. 
\end{definition}

\begin{corollary}
  As a category, $\mathbf{bg}_n$ is Abelian.
  \begin{proof}
    As a full subcategory of $\mathbf{qtUsc}_n$, we only have to check closedness under finite sums, kernels, and cokernels. We do the case of sums as an example. Let $V_\mathbf{B} \cong V$ and $W_\mathbf{C}  \cong W$ with $\mathbf{B,C}$ bounded. Let $\mathbf{D} = \mathbf{B}\cup\mathbf{C}$. Then \[(V\oplus W)_{\mathbf{D}} \cong V_\mathbf{D}\oplus W_\mathbf{D} \cong V\oplus W\] which implies that $\mathbf{bg}_n$ is closed under sums.
  \end{proof}
\end{corollary}

\begin{proposition}\label{bgsep}
  The space $\mathbf{bg}_n$ is separable if and only if the base field $\mathbb{k}$ is countable or $n=1$.
  \begin{proof}
First let $n=1$. In this case the finitely presentable modules all decompose into finitely many interval modules. This already implies separability independently of the cardinality of the base field. 

Now let $\mathbb{k}$ be countable. It suffices to show that the space of finitely presentable modules has a countable, dense subset. Any such module is pointwise finite-dimensional and defined on a finite grid. If the base field is countable, then for every structure map there is only a countable amount of linear maps to choose from. Restricting to modules with countable grid lengths thus yields a countable, dense subset. For the other direction, assume that $\mathbb{k}$ is uncountable and $n\geq 2$. In order to show that the resulting space is not separable, we produce an uncountable family of modules, with distance bounded from below by some $\varepsilon>0$. For this we consider the following well-known example from \cite[Chapter 8]{botnan23}: 

\[\begin{tikzcd}
	{\mathbb{k}} & {\mathbb{k}\oplus \mathbb{k}} & {\mathbb{k}\oplus \mathbb{k}} & {\mathbb{k}\oplus \mathbb{k}} \\
	& {\mathbb{k}} & {\mathbb{k}\oplus \mathbb{k}} & {\mathbb{k}\oplus \mathbb{k}} \\
	&& {\mathbb{k}} & {\mathbb{k}\oplus \mathbb{k}} \\
	&&& {\mathbb{k}}
	\arrow["\begin{array}{c} \begin{pmatrix}0\\ 1\end{pmatrix} \end{array}", from=1-1, to=1-2]
	\arrow["{\operatorname{id}}", from=1-2, to=1-3]
	\arrow["{\operatorname{id}}", from=1-3, to=1-4]
	\arrow["\begin{array}{c} \begin{pmatrix}1\\0\end{pmatrix} \end{array}", from=2-2, to=1-2]
	\arrow["\begin{array}{c} \begin{pmatrix}1\\ 0\end{pmatrix} \end{array}"', from=2-2, to=2-3]
	\arrow["{\operatorname{id}}"', from=2-3, to=1-3]
	\arrow["{\operatorname{id}}", from=2-3, to=2-4]
	\arrow["{\operatorname{id}}"', from=2-4, to=1-4]
	\arrow["\begin{array}{c} \begin{pmatrix}1\\ 1\end{pmatrix} \end{array}"', from=3-3, to=2-3]
	\arrow["\begin{array}{c} \begin{pmatrix}1\\ 1\end{pmatrix} \end{array}"', from=3-3, to=3-4]
	\arrow["{\operatorname{id}}"', from=3-4, to=2-4]
	\arrow["\begin{array}{c} \begin{pmatrix}1\\\lambda\end{pmatrix} \end{array}"', from=4-4, to=3-4]
\end{tikzcd}\]
For $\lambda\neq\mu\in\mathbb{k}$, the resulting modules are not isomorphic. With a given grid size, their interleaving distances are bounded from below. 
  \end{proof}
\end{proposition}

\begin{remark}
  In particular, the space of boundedly generated modules is separable for finite fields $\mathbb{F}_q$, for number fields $K\mid\mathbb{Q}$, and for function fields $F\mid\mathbb{F}_q(t)$ i.e., for global fields.
\end{remark}

Recall that a subset of a topological space is \emph{precompact} if its closure is compact.
In the above proposition, we used the module $\operatorname{im}(\eta_\varepsilon:V\to V[\varepsilon])$.
For $V\in\mathbf{qtOb}_n$ this module is not an observable invariant (i.e., not well-defined on isomorphism classes in the observable category), but the module \[S_\varepsilon V = \operatorname{im}(\underline{V}\to\underline{V}[\varepsilon])\] is, since  the lower envelope $\underline{\,\cdot\,}$ is a functor out of the observable category. 
We could also use the upper envelope functor $\overline{\,\cdot\,}$, or a mix of the two.
The result will always be the same up to observable isomorphism.

\begin{proposition}\label{prop:characterization-precompactness}
  A subset $A\subseteq \mathbf{qtOb}_n$ is precompact if and only if for all $\delta>0$ the set
  \[A\vert_\delta = \{S_\delta V\vert_{\delta\mathbb{Z}^n}:V\in A\}\]
   has finitely many isomorphism classes. 

  \begin{proof}
    We use the following well-known result:

    \begin{lemma}
      Let $X$ be a complete metric space. A set $A\subseteq X$ is precompact if and only if it is totally bounded. 
    \end{lemma}

    We observe that from the proof of \cref{lem:regularly-locally-fp-dense} we obtain \[d_I(V,S_\varepsilon V\vert_{\varepsilon\mathbf{Z}^n})\leq 2\varepsilon\] for any $V$.
    We first assume that every $A\vert_\delta$ has finitely many isomorphism classes. Fix $\varepsilon >0$ and let $0<\delta <\frac{\varepsilon}{2}$.
    Let $V_1,\dots,V_m$ be the isomorphism classes in $A_\delta$ extended to $\mathbb{R}^n$.
    For all $V\in A$ there is some $V_k$ so that $d_I(V,V_k)\leq 2\delta < \varepsilon$.
    We thus obtain \[A\subseteq\bigcup_{k=1}^m B_\varepsilon(V_k).\]
    Since $\varepsilon>0$ was arbitrary, it follows that $A$ is totally bounded.
    
    Conversely, let $A$ be totally bounded and fix $\delta>0$.
    We can then find a cover \[A\subseteq \bigcup_{k=1}^m B_{\frac{\delta}{2}}(V_k)\] with $V_1,\dots,V_m\in A$.
    This means that for any  $V\in A$ we have some $V_k$ with $d_I(V,V_k)<\frac{\delta}{2}$.
    In particular, there is a $\frac{\delta}{2}$-interleaving which gives us a diagram 
    \[\begin{tikzcd}
      {V\vert_{\delta\mathbf{Z}^n}} && {V[\delta]\vert_{\delta\mathbf{Z}^n}} \\
      & {V_k[\frac{\delta}{2}]\vert_{\delta\mathbf{Z}^n}}
      \arrow[from=1-1, to=1-3]
      \arrow[from=1-1, to=2-2]
      \arrow[from=2-2, to=1-3]
    \end{tikzcd}\]
    implying that $S_\delta V\vert_{\delta \mathbf{Z}^n}$ depends only on $V_k$. 
    It follows that $A\vert_\delta$ has at most $m$ isomorphism classes, proving the claim.
  \end{proof}
\end{proposition}

\begin{definition}
  Let $V\in\mathbf{qtOb}_n$.
  We call the map \[c^V:[0,\infty)\to  \mathbb{N} \cup \{\infty\} \] defined by \[c^V(\varepsilon) = \sup_s \operatorname{rk}\left( \underline V_s \to \overline V[\varepsilon]_s\right)\] the \emph{persistent rank} of $V$. 
\end{definition}

\begin{corollary}
  \label{corollary:conditions-precompactness-finite-field}
  Assume that $\mathbb{k}$ is finite. Then the following conditions on $A\subseteq\mathbf{qtOb}_n$ imply that $A$ is precompact:

  \begin{enumerate}
    \item The modules $V\in A$ have uniformly bounded support.
    \item The modules $V\in A$ have uniformly bounded persistent rank.  
  \end{enumerate}
\end{corollary}

\section{Applications}

\subsection{Sublevel Set Persistence}

Let $X$ be a topological space and let $f:X\to\mathbb{R}^n$ be a continuous function.
Then \[X_s = \{x\in X:f(x) \ll s\}\] defines a filtration of open subspaces of $X$.
By the continuity of $f$ we have \[X_s \cong \underset{\varepsilon>0}{\operatorname{colim}}~X_{s-\varepsilon} \cong \bigcup_{\varepsilon>0}X_{s-\varepsilon}.\]
The assignment \[s\mapsto H_k(X_s,\mathbb{k})\] defines a $n$-parameter persistence module for fixed $k\geq 0$.
Singular homology preserves filtered colimits (see for example \cite{may99}, chapter 14), so \[H_k(X_s,\mathbb{k}) \cong H_k(\underset{\varepsilon>0}{\operatorname{colim}}~ X_{s-\varepsilon},\mathbb{k}) \cong \underset{\varepsilon>0}{\operatorname{colim}}~H_k(X_{s-\varepsilon},\mathbb{k})\] and hence the module $H_k(X_\bullet,\mathbb{k})$ is lower semicontinuous.
The next proposition is similar in spirit to \cite[Theorem~2]{cagliari-landi} and  \cite[Theorem~3.33]{strandstab}.

\begin{proposition}
  If $X$ is finitely triangulable, then the module $H_k(f) = H_k(X_\bullet,\mathbb{k})$ is in $\mathbf{bg}_n$. 

  \begin{proof}
    We first consider q-tameness. Let $s\ll t$.
    We want to show that $H_k(X_s,\mathbb{k})\to H_k(X_t,\mathbb{k})$ has finite image.
    We triangulate $X$ and refine the triangulation until no simplex intersects both $f^{-1}(s)$ and $f^{-1}(t)$.
    Let $K$ be the union of simplices intersecting $X_s$.
    Then $X_s\subseteq K\subseteq X_t$ which provides a factorization \[H_k(X_s,\mathbb{k})\to H_k(K,\mathbb{k})\to H_k(X_t,\mathbb{k})\] which shows the claim since $K$ is the geometric realization of a finite simplicial complex.
    We already know that $V = H_k(f)$ is lower semicontinuous from the above discussion. Thus it remains to show that it is the restriction-extension over a bounded set in $\mathbf{R}^n$.  For this we observe that $\mathbf{B} = f(X)$ is compact in $\mathbf{R}^n$ and hence bounded. We claim that $V = V_\mathbf{B}$. For this we observe that $X_s = f^{-1}(\{p : p\ll s\}\cap \mathbf{B}) = X_t$ where $t=\sup\{p : p\in \mathbf{B},p\ll s\}$, which is in $\mathbf{B}$ by closedness. Thus, as a persistence set, the filtration $X_\bullet$ is the restriction-extension of itself over $\mathbf{B}$, hence the same also holds for $V$. 
  \end{proof}
\end{proposition}

The results from \cref{section:decomposition} then immediately give us the following corollary:

\begin{corollary}
  Let $X$ be a finitely triangulable space and let $X_\bullet$ be any continuous $n$-parameter filtration into open subspaces.
  Then the module $H^k(X_\bullet,\mathbb{k})$ decomposes essentially uniquely into indecomposable modules.  
\end{corollary}

Let $X$ be a finitely triangulable space.
Then for any two continuous functions $f,g:X\to\mathbb{R}^n$ we have \[d_I(H_k(f),H_k(g))\leq \|f-g\|_\infty,\] which implies that the mapping \[C(X,\mathbb{R}^n)\to \mathbf{qtOb}_n\] is continuous. 

\begin{proposition}
  Let $X$ be a finitely triangulable space and let $A\subseteq C(X,[0,1]^n)$ be a uniformly Lipschitz and closed family of continuous functions.
  Then the image of $A$ in $\mathbf{qtOb}_n$ is compact.

  \begin{proof}
    We use the theorem of Arzelà-Ascoli:

    \begin{theorem}
      Let $X$ be a compact Hausdorff space. A subset $A\subseteq C(X,\mathbb{R}^n)$ is relatively compact if and only if it is equicontinuous and pointwise bounded.
    \end{theorem}

    A proof of this version can be found in \cite{kelley1976linear}.
    By considering each component separately, this generalizes to functions with values in $\mathbb{R}^n$. The result then follows from the fact that the image of compact sets is compact. 
  \end{proof}
\end{proposition}

\begin{remark}
  From \cref{prop:characterization-precompactness} it now follows that for a uniformly Lipschitz and closed family $A\subseteq C(X,\mathbb{R}^n)$ the set $A\vert_\delta$ has finitely many isomorphism classes for all $\delta>0$.
\end{remark}

\begin{proposition}
  Let $X$ be a finitely triangulable space. Then the image of $C(X,\mathbb{R}^n)$ in $\mathbf{bg}_n$ is separable. 

  \begin{proof}
    Let $\mathbb{l}$ be the prime field of $\mathbb{k}$, in particular $\mathbb{l}$ is countable.
    Let $f:X\to\mathbb{R}^n$ be continuous.
    Then by the universal coefficient theorem for homology (see for example \cite{hatcher2002algebraic}) we have \[H_k(X_s,\mathbb{k} ) \cong H_k(X_s,\mathbb{l})\otimes_{\mathbb{l}} \mathbb{k}\] and so $H_k(f,\mathbb{k}) \cong H_k(f,\mathbb{l})\otimes_\mathbb{l}\mathbb{k}$.
    We therefore have a factorization \[H_k(\cdot,\mathbb{k}):C(X,\mathbb{R}^n)\to \mathbf{bg}_n^\mathbb{l}\to\mathbf{bg}_n^\mathbb{k}\] where the last map is given by $V\mapsto V\otimes_\mathbb{l}\mathbb{k}$.
    We have \[d_I(V\otimes_\mathbb{l} \mathbb{k},W\otimes_\mathbb{l}\mathbb{k})\leq d_I(V,W)\] for modules with coefficients in  $\mathbb{l}$ since tensoring preserves interleaving diagrams which imples the continuity of the tensor operation.
    By \cref{bgsep}, we know that $\mathbf{bg}_n^\mathbb{l}$ is separable.
    It follows by continuity that the image of $C(X,\mathbb{R}^n)$ in $\mathbf{bg}_n^\mathbb{k}$ is also separable.
  \end{proof}
\end{proposition}

\subsection{Generic Indecomposability}

One of the main motivations for this work were the results on density of indecomposable multi-parameter persistence modules in \cite{bauer23}.
With the space of q-tame and upper semicontinuous modules, we now indeed have a Baire space for multiparameter persistence. However, the theorem on the density of indecomposables does not immediately generalize to this class of modules:

\begin{proposition}
  The indecomposables are not dense in $\mathbf{qtUsc}_n$.

\begin{proof}
Write $L$ for the module that is $\mathbb{k}$ everywhere and has only identities as structure maps.
It is straightforward to verify that the module $L$ is injective.
We show that $L\oplus L$ has no indecomposable in any $\varepsilon$-neighbourhood in the interleaving distance. 
We now consider an $\varepsilon$-interleaving $f:L\oplus L\to V[\varepsilon],g:V\to (L\oplus L)[\varepsilon]$.
We then have in particular the commutative diagram 
\[\begin{tikzcd}
	{L\oplus L} && {(L\oplus L)[2\varepsilon]} \\
	& {V[\varepsilon]}
	\arrow["\cong", from=1-1, to=1-3]
	\arrow["f"', from=1-1, to=2-2]
	\arrow["{g[\varepsilon]}"', from=2-2, to=1-3]
\end{tikzcd}\] which implies that $f:L\oplus L\to V[\varepsilon]$ is injective.
By composing with one of the inclusions $L\to L\oplus L$ we get a short exact sequence  
\[\begin{tikzcd}
	0 & L & V & {V/L} & 0
	\arrow[from=1-1, to=1-2]
	\arrow[from=1-2, to=1-3]
	\arrow[from=1-3, to=1-4]
	\arrow[from=1-4, to=1-5]
\end{tikzcd}\] which splits since $L$ is injective, in particular $V = L\oplus V/L$.
We also have $V/L \neq 0$ since otherwise we could factorize the identity $\mathbb{k}\oplus \mathbb{k}\to \mathbb{k}\oplus\mathbb{k}$ over a $1$-dimensional space.
Hence any $V$ which is $\varepsilon$-interleaved with $L\oplus L$ for any $\varepsilon>0$ cannot be indecomposable. 
\end{proof}

\end{proposition}

In light of this result we restrict ourselves to the smaller space of boundedly generated modules. 
Recall that $\overline{\mathbf{fp}_n}  =\mathbf{bg}_n$, so it is immediate that the indecomposables are dense by \cite[Theorem A]{bauer23}.
Recall that a \emph{residual set} is a countable intersection of open and dense sets, a \emph{Baire space} is a topological space in which any residual set is dense (which is the case for every complete metric space), and that a \emph{generic property} in a Baire space is a property that holds on a residual set.
Now the missing ingredient for genericity is the construction of a family open and dense sets whose intersection is the set of indecomposables.
To this end we show the next proposition:

\begin{proposition}\label{prop:open-indecomp}
  Let $V$ be a regularly locally finitely presentable and indecomposable persistence module.
  Then there is a constant $\mu>0$ so that every $W\in\mathbf{qtUsc}_n$ with $d_I(V,W)<\mu$ is $c\mu$-indecomposable, where $c>0$ is fixed. 
\end{proposition}

Since the proof is a straightforward generalization from the proof of \cite[Theorem B]{bauer23}, we relegate it to \cref{section:openness-indecomposables}.
We conclude that the indecomposables are not only dense, but also form a residual subset of $\mathbf{bg}_n$.

\begin{proposition}
  Let $n\geq 2$. In the space $\mathbf{bg}_n$, being indecomposable is a generic property.

 \begin{proof}
   Observe that \cref{prop:open-indecomp} shows that there is an open set in $\ourspace{n}$ containing all indecomposables and containing only $\varepsilon$-indecomposables for some $\varepsilon>0$. The intersection with $\mathbf{bg}_n$ becomes dense in $\mathbf{bg}_n$ and remains open in the subspace topology.  
   By shrinking the balls around the indecomposables we get a countable sequence of open and dense sets whose intersection is precisely the set of indecomposables.
 \end{proof}
\end{proposition}

\subsection{Compactness of Degree-Rips Constructions}

Prominent examples of density-sensitive bifiltrations are the Degree-Rips and the Subdivision-Rips filtrations.
In this section we show that certain families of persistence modules arising this way are precompact.
The construction of these bifiltrations can be found in \cite{blumberg2022stability2parameterpersistenthomology}.

\begin{proposition}\label{prop:precompactness-finite-samples}
  Let $X$ be a compact metric space.
  The set of finite samples of $X$ (endowed with the normalized counting measure)
  is precompact under $d_{\mathrm{GHP}}$.


  \begin{proof}
     In \cite[Proposition 8]{miermont2009tessellationsrandommapsarbitrary} it is shown that a set of compact metric probability spaces is precompact under $d_{\mathrm{GHP}}$ if and only if it is precompact under $d_{\mathrm{GH}}$ when we forget the probability measure.
    Now precompactness of the set of finite samples of $X$ under $d_{\mathrm{GH}}$ follows from Gromov's precompactness Theorem \cite{gromov,mgeom}, since the finite samples of $X$ have uniformly bounded diameter and are uniformly totally bounded.
  \end{proof}
\end{proposition}


\begin{proposition}\label{prop:stability-degree-subdiv-rips}
  The Degree-Rips and the Subdivision-Rips bifiltration are stable in the sense that they define Lipschitz maps from the space of finite metric probability spaces endowed with the normalized counting measure to the space of persistence modules, when postcomposed with homology.

  \begin{proof}
    The case of Degree-Rips follows from \cite[Corollary 6.5.2]{scoccola}.
    The stability of the Subdivision-Rips bifiltration was shown by Blumberg and Lesnik  in Theorem 1.6. (iii) of \cite{blumberg2022stability2parameterpersistenthomology}.
  \end{proof}
\end{proposition}

We therefore obtain the following corollary:

\begin{corollary}
  Let $X$ be a compact metric space and let $S$ be the set of finite samples of $X$ endowed with the normalized counting measure.
  Let $F$ denote either the Degree-Rips or the Subdivision-Rips constructions postcomposed with homology with field coefficients.
  Then $F(S) \subseteq \mathbf{qtUsc}_n$ is precompact.
\begin{proof}
  By \cref{prop:stability-degree-subdiv-rips}, the map $F$ is Lipschitz and hence continuous. By \cref{prop:precompactness-finite-samples}, the set $S$ is precompact, and so is the image $F(S)$.
  It remains to be shown that $F(S)$ indeed lands in $\mathbf{qtUsc}_n$, that is, if $Y$ is a finite metric space endowed with the normalized counting measure, then $F(Y)$ is q-tame and upper semicontinuous.
  The fact that it is q-tame is clear since all simplicial complexes used in the construction of $F$ are finite.
  The fact that it is upper semicontinuous is a straightforward consequence of the fact that the construction of $F$ uses non-strict inequalities \cite{blumberg2022stability2parameterpersistenthomology}.
\end{proof}
\end{corollary}

\printbibliography{}

\appendix\section{Openness of Nearly Indecomposables}\label{section:openness-indecomposables}

For the convenience of the reader we give the proof of proposition 4.12. It is completely analogous to the proof of \cite[Theorem B]{bauer23}.
Both of the lemmas we use are also taken from this paper.

\begin{proposition}
  Let $V$ be a regularly locally finitely presentable and indecomposable module.
  Then there is a constant $\mu>0$ so that every $W\in\mathbf{qtUsc}_n$ with $d_I(V,W)<\mu$ is $c\mu$-indecomposable where $c>0$ is fixed. 
  \begin{proof}
    Let $V$ be defined on a grid $\mathbf{P} = \varepsilon\mathbf{Z}^n$.
    Let $d_I(V,W)<\frac{\varepsilon}{2}$.
    Let $f:V\to W[\delta],g:W\to V[\delta]$ be a $\delta$-interleaving with $\delta<\frac{\varepsilon}{2}$.
    We then have the commutative diagram 
    \[\begin{tikzcd}
      V && {V[2\delta]} \\
      & {W[\delta]}
      \arrow["{\eta^V_{2\delta}}", from=1-1, to=1-3]
      \arrow["f"', from=1-1, to=2-2]
      \arrow["{g[\delta]}"', from=2-2, to=1-3]
    \end{tikzcd}\] and by the functoriality of restriction-extensions 
    \[\begin{tikzcd}
      {V_{\mathbf{P}}} && {V[2\delta]_{\mathbf{P}}} \\
      & {W[\delta]_{\mathbf{P}}}
      \arrow["{(\eta^V_{2\delta})_\mathbf{P}}", from=1-1, to=1-3]
      \arrow["{f_{\mathbf{P}}}"', tail, from=1-1, to=2-2]
      \arrow["{g[\delta]_{\mathbf{P}}}"', two heads, from=2-2, to=1-3]
    \end{tikzcd}\] since the shift becomes an isomorphism.
    The short exact sequence 
    \[\begin{tikzcd}
      0 & {\ker g[\delta]_\mathbf{P}} & {W[\delta]_\mathbf{P}} & {V[2\delta]_\mathbf{P}} & 0
      \arrow[from=1-1, to=1-2]
      \arrow[from=1-2, to=1-3]
      \arrow["{g[\delta]}", from=1-3, to=1-4]
      \arrow[from=1-4, to=1-5]
    \end{tikzcd}\] splits via $f\circ (\eta^V_{2\delta})_\mathbb{P}^{-1}$ and so we have \[W[\delta]_\mathbf{P} \cong V_\mathbf{P}\oplus X \cong V\oplus X\] with $X = \ker g[\delta]_\mathbf{P}$.
    There is a decomposition \[W \cong \bigoplus_{i\in I}W_i\] into indecomposables by our results from the third chapter.
    We then have \[W[\delta]_\mathbf{P} \cong \bigoplus_{i\in I}W_{i}[\delta]_{\mathbf{P}}\] and since $V$ is indecomposable there is some $W_i$ with $(W_i[\delta])_{\mathbf{P}} \cong V\oplus Y$.
    If we let $A = W_i$ and $B = \bigoplus_{j\in I\setminus\{i\}} W_j$ we have a decomposition $W = A\oplus B$ and it remains to show that $B$ is strictly trivial for a suitable constant.
    We first show that $B_\mathbf{P}$ is $\sigma$-trivial for $\sigma>\varepsilon$.
    For this observe that $B_\mathbf{P}$ is a direct summand of $X$ so it suffices to consider $X$.
    First we see that by naturality 
    \[\begin{tikzcd}
      {V\oplus X} & {V[\sigma]\oplus X[\sigma]} \\
      {W[\delta]_\mathbf{P}} & {W[\delta]_\mathbf{P}[\sigma]}
      \arrow["{(\eta^V_\sigma,\eta^X_\sigma)}", from=1-1, to=1-2]
      \arrow["\cong"', from=1-1, to=2-1]
      \arrow["\cong", from=1-2, to=2-2]
      \arrow["{\eta_\sigma^{W[\delta]_\mathbf{P}}}"', from=2-1, to=2-2]
    \end{tikzcd}\] commutes for all $\sigma >0$.
    By the interleaving we have that 
    \[\begin{tikzcd}
      {W[\delta]_\mathbf{P}} && {W[3\delta]_\mathbf{P}} \\
      & {V[2\delta]_\mathbf{P}}
      \arrow["{(\eta_{2\delta}^{W[\delta]})_\mathbf{P}}", from=1-1, to=1-3]
      \arrow["{g[\delta]_\mathbf{P}}"', from=1-1, to=2-2]
      \arrow[from=2-2, to=1-3]
    \end{tikzcd}\] commutes, so $(\eta^{W[\delta]}_{2\delta})_{\mathbf{P}}$ vanishes on $X$. We need a lemma:

    \begin{lemma}
      Let $0<\alpha <\beta$ and consider a countable grid $\mathbf{Q}$ with widths in $[\alpha,\beta]$.
      Let $L$ be a module.
      For $r\leq \alpha$ there is a morphism making 
      \[\begin{tikzcd}
        {L_\mathbf{Q}} & {L_{\mathbf{Q}}[\beta]} \\
        {L[r]_\mathbf{Q} = L_{\mathbf{Q}+r}[r]}
        \arrow["{\eta_\beta^{L_\mathbf{Q}}}", from=1-1, to=1-2]
        \arrow["{(\eta^L_r)_\mathbf{Q}}"', from=1-1, to=2-1]
        \arrow[dotted, from=2-1, to=1-2]
      \end{tikzcd}\] commute.
      
      \begin{proof}
        We first define the desired morphism, which is equivalently written as  \[m:L_{\mathbf{Q}+r}\to L_\mathbf{Q}[\beta-r].\]
        For $s\in\mathbf{R}^n$ we have \[L_{\mathbf{Q}+r}(s) = L(\sup\{t+r:t\in\mathbf{Q},t+r\leq s\})\] and \[L_\mathbf{Q}[\beta-r](s) = L_\mathbf{Q}(s+\beta-r) = L(\sup\{t:t\in\mathbf{Q},t\leq s+\beta-r\}).\] 
        We now let $s_0$ be so that $t_0+r = \sup\{t+r:t\in\mathbf{Q},t+r\leq s\}$ so that $L_{\mathbf{Q}+r}(s) = L(t_0+r)$.
        Writing $\mathbf{Q} = (\{t_i\}_{i\in\mathbf{Z}})^n$ we have $t_0 = (t_{i_1},\dots,i_{i_n})$ for suitable indices $i_\bullet$.
        Write $s_1 = (s_{i_1+1},\dots,s_{i_n+1})$.
        Then since our grid has mesh widths in $[\alpha,\beta]$ we get $t_1-\beta\leq t_0$ and so $t_1-\beta+r\leq t_0+r\leq s$ which in turn implies that $t_1\leq s+\beta-r$.
        Moreover, \[t_1\leq \sup\{t:t\in\mathbf{Q},t\leq s+\beta-r\}\] and by $r\leq \alpha$ we obtain $t_0+r\leq t_1$.
        This allows us to consider the composition \[m_s:L_{\mathbf{Q}+r}(s) = L(t_0+r)\to L(t_1)\to L(\sup\{t:t\in\mathbf{Q},t\leq s+\beta-r\}) = L[\beta-r](s)\] given by structure morphisms of $L$.
        Since these are constructed from structure maps only we immediately obtain naturality and commutativity of the diagram.
      \end{proof}
    \end{lemma}

    We apply the lemma to our $\varepsilon$-grid with $r = 2\delta <\varepsilon$ and $L = W[\delta]$ to obtain a factorization 
    \[\begin{tikzcd}
      {W[\delta]_\mathbf{P}} & {W[\delta]_\mathbf{P}[\sigma]} \\
      {W[3\delta]_\mathbf{P}}
      \arrow[from=1-1, to=1-2]
      \arrow[from=1-1, to=2-1]
      \arrow[dotted, from=2-1, to=1-2]
    \end{tikzcd}\] for any $\sigma >\varepsilon$. Thus $\eta^X_\sigma$ vanishes and $X$ is strictly $\sigma$-trivial. We now need another lemma:

    \begin{lemma}
      Let $\alpha,\beta> 0$ and let $\mathbf{Q}$ be a countable grid with widths $\leq \beta$.
      Let $L$ be a module. If $L_\mathbf{Q}$ is (strictly) $\alpha$-trivial then $L$ is (strictly) $(\alpha+\beta)$-trivial.

      \begin{proof}
        We show the non-strict case; the strict case follows analogously.
        We want to show that \[L(s)\to L(s+(\alpha+\beta))\] vanishes for all $s\in\mathbf{R}^n$.
        We can find a minimal $t$ with $s\leq t$. By our assumption on mesh widths we have $t\leq s+\beta$.
        Hence \[t+\alpha\leq s+ (\alpha+\beta).\]
        Write $u = \sup\{v\in\mathbf{Q}:v\leq t+\alpha\}$, then $L(t) = L_{\mathbf{Q}}(t)\to L_{\mathbf{Q}}(t+\alpha) = L(u)$ vanishes since $L_\mathbf{Q}$ is $\alpha$-trivial.
        We must have \[t\leq u\leq t+\alpha\] since $t$ is on $\mathbf{Q}$.
        We can thus conclude that the composition \[L(s)\to L(t)\to L(u)\to L(t+\alpha)\to L(s+(\alpha+\beta))\] vanishes, as desired.  
      \end{proof}
    \end{lemma}

    This shows that $B$ is $\sigma+\varepsilon> 2\varepsilon$-trivial for every $\sigma>\varepsilon$.
    We conclude that with $\mu = \frac{\varepsilon}{2}$ the module $W$ is strictly $6\mu = 3\varepsilon$-trivial since we can choose $\sigma <2\varepsilon$.
  \end{proof}
\end{proposition}

\end{document}